\documentclass[a4paper]{amsart}
\usepackage{ricky}

\title[Eigenvarieties for cuspforms over PEL type Shimura varieties]{Eigenvarieties for cuspforms over PEL type Shimura varieties with dense ordinary locus}

\author{Riccardo Brasca}
\thanks{This work was performed within the framework of the LABEX MILYON (ANR-10-LABX-0070) of Université de Lyon, within the program ``Investissements d'Avenir'' (ANR-11-IDEX-0007) operated by the French National Research Agency (ANR)}

\email{\href{mailto:riccardo.brasca@gmail.com}{riccardo.brasca@gmail.com}}

\urladdr{\url{http://www.imj-prg.fr/~riccardo.brasca/}}

\address{Institut de Mathématiques de Jussieu - Paris Rive Gauche\\
Université Paris Diderot\\
Paris\\
France}



\begin{document}

\begin{abstract}
Let $p > 2$ be a prime and let $X$ be a compactified PEL Shimura variety of type (A) or (C) such that $p$ is an unramified prime for the PEL datum and such that the ordinary locus is dense in the reduction of $X$. Using the geometric approach of Andreatta, Iovita, Pilloni, and Stevens we define the notion of families of overconvergent locally analytic $p$-adic modular forms of Iwahoric level for $X$. We show that the system of eigenvalues of any finite slope cuspidal eigenform of Iwahoric level can be deformed to a family of systems of eigenvalues living over an open subset of the weight space. To prove these results, we actually construct eigenvarieties of the expected dimension that parametrize finite slope systems of eigenvalues appearing in the space of families of cuspidal forms.
\end{abstract}

\maketitle

\section*{Introduction} \label{sec: intro}
The main theme of this work is the theory of $p$-adic families of modular eigenforms and the study of the congruences between them. This subject has become more and more important in number theory. For example, one of the main technique to prove modularity results of Galois representations is to prove that a given Galois representation lives in a family of Galois representations attached to a family of $p$-adic modular forms and then invoke a classicity result. To achieve these goals, it is crucial to have a good understanding of $p$-adic families of modular forms.

Historically, the subject started in the seventies with the work of Serre that gave the first example of a $p$-adic family of eigenforms for $\GL_{2/\Q}$: the Eisenstein family. In the eighties Hida was able to prove that, still in the $\GL_{2/\Q}$ case, \emph{any} ordinary eigenform (of Iwahoric level) can be deformed to a family of ordinary eigenform over the weight space. During the nineties, Coleman was able to generalize Hida's result to forms that are overconvergent and of finite slope for the $\U$-operator. The theory culminated in the construction, due to Coleman and Mazur, of the \emph{eigencurve}, a rigid analytic curve living over the weight space that parametrize finite slope overconvergent eigenforms.

It is then natural to try to generalize the theory to groups different from $\GL_{2/\Q}$. Coleman's techniques are based on the theory of $q$-expansion and on the existence of the Eisenstein family, so it seems difficult to generalize them to more general groups (where we can have no cusps or no Eisenstein families). There are already several other approaches in the literature. For example in the work of Kisin and Lai for Hilbert modular forms in \cite{kislai} (recently generalized by Mok and Tan to the Siegel-Hilbert case in \cite{sieghilb}) the authors use a generalization of the Eisenstein family and construct the eigenvariety following Coleman and Mazur. More generally in \cite{urban_eigen}, Urban uses Steven's theory of overconvergent cohomology to show the existence of an eigenvariety for modular symbols associated to any reductive group with discrete series. We also have results by Chenevier for unitary groups in \cite{chenevier} or by Emerton in \cite{emerton} for any reductive group. 

In all these constructions, the eigenvariety parametrizes systems of eigenvalues appearing in the space of overconvergent modular forms rather than modular forms themselves. One of the reasons for this is the lack of the notion of families of overconvergent modular forms (while the notion of family of systems of eigenvalues is easily defined). In this paper, the starting point to build the eigenvariety is the definition of analytic families of overconvergent modular forms. We follow the geometric approach recently introduced by Andreatta, Iovita, Pilloni, and Stevens in the series of papers \cite{over, vincent, AIP, AIPhilb}. The basic idea is quite simple: analytically interpolate the sheaves $\underline \omega^{\underline k}$, where $\underline k$ is an integral weight, defining the sheaves $\underline \omega^\chi$ for any $p$-adic weight $\chi$. More generally, one wants to define a family of sheaves parametrized by the weight space with the property that its pullback to a point $\chi$ of the weight space is the sheaf $\underline \omega^\chi$. Since we are interested in overconvergent modular forms, such a family should be a sheaf $\underline \omega^{\mc U}$ over $X(v) \times \mc U$, where $X(v)$ is a sufficiently small strict neighbourhood of the ordinary locus of the relevant (compactified) Shimura variety and $\mc U$ is an affinoid of the weight space. Then, a family of modular forms parametrized by $\mc U$ is simply a global section of $\underline \omega^{\mc U}$. After having defined Hecke operators, the idea is to use the abstract machinery developed by Buzzard in \cite{buzz_eigen} (that generalizes Coleman's work) to construct the eigenvariety. Unfortunately, we do not know whether one crucial assumption in Buzzard's work is verified by the space of families of modular forms (and we believe that in general it is not), but we are able to show that this assumption is satisfied by \emph{cuspidal} forms. Once this is done Buzzard's results apply and we obtain the eigenvariety.

Let us now state more precisely the results obtained in this paper. Let $K$ be a sufficiently large finite extension of $\Q_p$ and let $Y$ be a Shimura variety, over $K$, of PEL type and Iwahoric level, associated to a symplectic or a unitary group. We assume that $p > 2$ is a prime that is unramified in the PEL datum of $X$ and let $\mc W$ be the weight space associated to $Y$. We denote with $X$ a fixed smooth toroidal compactification of $Y$. We assume that the ordinary locus of the reduction of a certain integral model of $X$ modulo the maximal ideal of $\mc O_K$ is dense (see \cite{shimura} for a case without ordinary locus). Our main results are the following theorems.
\begin{teono}
Let $\chi \in \mc W$ be a $p$-adic character. There is a good notion of $\underline v$-overconvergent, $\underline w$-locally analytic modular forms over $X$, where $\underline v$ and $\underline w$ are tuples of positive rational numbers satisfying certain conditions. These modular forms are defined as sections of certain sheaves $\underline \omega_{\underline v, \underline w}^\chi$ that interpolate an analytic version of the classical algebraic sheaves $\underline \omega^{\underline k}$ defined for integral weights. We also have an analogous result for cuspforms. These spaces can be put in families over $\mc W$ and there is an action of an Hecke algebra $\m T$, that includes the completely continuous operator $\U$. If $F$ is a locally analytic overconvergent modular eigenform of integral weight and $\U$-slope sufficiently small with respect to the weight (see Theorem~\ref{thm: class} for a precise condition), then $F$ comes from a classical modular form.
\end{teono}
\begin{teono}
Let $\cusp^{\dagger\chi}_{\underline v,\underline w}$ be the space of cuspidal forms that are $\underline v$-overconvergent, $\underline w$-locally analytic and of weight $\chi$. Let $f \in \cusp^{\dagger\chi}_{\underline v,\underline w}$ be a cuspidal eigenform of finite slope for the $\U$-operator. Then there exists an affinoid $\mc U \subset \mc W$ that contains $\chi$ and such that the system of eigenvalues associated to $f$ can be deformed to a family of systems of eigenvalues appearing in $\cusp^{\dagger\mc U}_{\underline v,\underline w}$, where $\cusp^{\dagger\mc U}_{\underline v,\underline w}$ is the space of families of cuspforms parametrized by $\mc U$.

More precisely, there is a rigid space $\mc E_{\underline v, \underline w} \to \mc W \times \m A^{1,\rig}$ that satisfies the following properties.
\begin{enumerate}
 \item It is equidimensional of dimension $\dim(\mc W)$ and the map $\mc E_{\underline v, \underline w} \to \mc W$ is locally finite. The fiber of $\mc E_{\underline v, \underline w}$ above a point $\chi \in \mc W$ parametrizes systems of eigenvalues for the Hecke algebra $\m T$ appearing in $\cusp^{\dagger\chi}_{\underline v,\underline w}$ that are of finite slope for the $\U$-operator. If $x \in \mc E_{\underline v, \underline w}$, then the inverse of the $\U$-eigenvalue corresponding to $x$ is $\pi_2(x)$, where $\pi_2$ is the induced map $\pi_2 \colon \mc E_{\underline v, \underline w} \to \m A^{1,\rig}$. For various $\underline v$ and $\underline w$, these constructions are compatible. Letting $\underline v \to 0$ and $\underline w \to \infty$ we obtain the global eigenvariety $\mc E$.
 \item Let $f \in \cusp^{\dagger\chi}_{\underline v,\underline w}$ be a cuspidal eigenform of finite slope for the $\U$-operator and let $x_f$ be the point of $\mc E_{\underline v, \underline w}$ corresponding to $f$. Let us suppose that $\mc E_{\underline v, \underline w} \to \mc W$ is unramified at $x_f$. Then there exists an affinoid $\mc U \subset \mc W$ that contains $\chi$ and such that $f$ can be deformed to a family of finite slope eigenforms $F \in \cusp^{\dagger\mc U}_{\underline v,\underline w}$.
\end{enumerate}
\end{teono}
Here is a detailed description of this paper. We follow \cite{AIP}, that is our main reference.

In Section~\ref{sec: PEL data} we introduce the Shimura varieties $X$ we work with. These are (integral models of) Shimura varieties of PEL type. At the beginning, we do not assume that $p$ is unramified in the PEL datum, but we assume that the ordinary locus of (the reduction of) $X$ is dense. We define Hasse invariants and some strict neighborhoods $X(\underline v)$ of the ordinary locus. We also work with some Shimura varieties of deeper level at $p$, that are needed to define modular forms. In Section~\ref{sec: F} we define the sheaf $\mc F$, that is a more convenient integral model of the conormal sheaf $\underline \omega$ and it is crucial for the definition of the sheaves $\underline \omega^\chi$. We also introduce our weight space and we define the so called modular sheaf of any $p$-adic weight (modular forms will be sections of these sheaves). In Section~\ref{sec: mod forms} we introduce various spaces of modular forms. We do not have a Koecher principle for sections of our modular sheaves, so we find it convenient to work with the compactified Shimura variety. In particular we need to assume that $p$ is unramified in the PEL datum. Section~\ref{sec: hecke op} is devoted to the definition of Hecke operators, both outside $p$ and at $p$. In particular we define the $\U$-operator and we show that it is a completely continuous operator on the space of overconvergent modular forms. In Section~\ref{sec: heck var} we study the space of cuspidal forms and we construct the eigenvarieties. To use Buzzard's machinery we need to verify that the space of cuspforms is projective (see Definition~\ref{defi: proj}) and to achieve this goal we make some explicit computations very similar to those of \cite{AIP}. The main technical point is a vanishing result about the higher direct images, projecting from the toroidal to the minimal compactification, of the structural sheaf twisted by the ideal defining the boundary. This result has been essentially proved in full generality by Lan in \cite{lan_ram}. Knowing the projectivity of the space of cuspforms we can apply Buzzard's results and we obtain the eigenvariety. In Section~\ref{sec: class} we prove our classicity result. This is a `small slope eigenforms are classical' type theorem, and its proof splits in two parts. First of all we show that small slope locally analytic overconvergent eigenforms are overconvergent algebraic eigenforms. This a representation theoretic computation using the BGG resolution. Then we use the results of \cite{stefan} to prove classicity. The above theorems follow.
\subsection*{Open questions} \label{subsec: open}
There are at least two questions left open by this work.
\begin{itemize}
 \item Since Corollary~\ref{coro: fin} holds only for cuspidal forms, we cannot apply Buzzard machinery to produce eigenvarieties that parametrize system of eigenvalues associated to not necessarily cuspidal modular forms. For various applications, it would be useful to have eigenvarieties that work in general. We think that the cuspidality assumption is necessary to apply Buzzard's machinery in general, but we believe that our construction can be used in some non-cuspidal cases, adding certain conditions on the weight. This is the subject of the work in progress \cite{eigen}.
 \item In view of Theorem~\ref{teo: eigen}, part \eqref{en: unr}, it is natural to look for some conditions that ensure that the morphism $\mc E \to \mc W$ is unramified at a given $x \in \mc E$. A similar problem exists already in the Siegel and Hilbert cases, see \cite[Section~8.3]{AIP} and the references cited there for what is know in the $\GL_{2}$ and $\GSp_4$ cases.
\end{itemize}
\subsection*{Acknowledgments} \label{subsec: ackn}
I would like to thank Fabrizio Andreatta, Adrian Iovita, and Vincent Pilloni for their paper \cite{AIP} and for many useful discussions. I would also like to thank Kai-Wen Lan for answering a lot of questions concerning his works. I thank the anonymous referee for pointing out a problem in the proof of Proposition~\ref{prop: mod form well def} and several useful comments.
\begin{notation}
If $G$ is an abelian group and $p$ is a prime number, we set $G_p \colonequals G \otimes_{\Z} \Z_p$.

If $R$ is a commutative ring, there is an equivalence of categories between $\mat_n(R)-\mathbf{mod}$ and $R-\mathbf{mod}$. We will always realize this Morita's equivalence via $M \rightsquigarrow e_{1,1} \cdot M$, where $e_{1,1} \in \mat_n(R)$ is the diagonal matrix that has $1$ in the upper left corner and $0$ on all the other entries. In the case of a module over a product of matrix rings, we will realize Morita's equivalence via the product of the above functors.

We will write $\B_n(R) \subset \GL_n(R)$ for the Borel subgroup of $\GL_n(R)$ consisting of upper triangular matrices. We denote with $\U_n(R)$ the unipotent radical of $\B_n(R)$.

We will work with several objects that can have a $+$ or a $-$ as superscripts, or no superscripts at all. If $\star$ is any symbol, the notation $\star^\pm$ refers to any of $\star^+$, $\star^-$ or $\star$. No ambiguity should arise.
\end{notation}

\section{PEL type Shimura varieties} \label{sec: PEL data}
In this section we introduce the basic objects of our work. Our main reference for Shimura varieties of PEL type is \cite{lan}. We consider a particular case of the situation studied in Lan's work but, until Section~\ref{sec: heck var}, we slightly relax the assumption on $p$. One can check that the definitions and results we cite still make sense with our assumptions on $p$, so we freely cite \cite{lan}.

Let $B$ be a finite dimensional simple algebra over $\Q$, with center $F$. We let $^\ast \colon B \to B$ be a positive involution and we write $F_0$ for the subfield of $F$ fixed by $^\ast$.
\begin{ass} \label{ass: type}
We assume that we are in one of the following situations.
\begin{enumerate}
 \item[Case (A)] We have $[F:F_0]=2$. In this case $F$ is a totally imaginary extension of $F_0$ and $B \otimes_{F_0} \mathds{R} \cong \mat_n(\C)$ (the involution is $A \mapsto \bar A^t$).
 \item[Case (C)] We have $F=F_0$ and $B \otimes_{F_0} \mathds{R} \cong \mat_n(\mathds{R})$ with involution $A \mapsto A^t$.
\end{enumerate}
\end{ass}
Let $d$ be $[F_0 : \Q]$ and let $\tau_1,\ldots,\tau_d$ be the various embeddings $F_0 \hookrightarrow \mathds{R}$. In case (A), we choose once and for all a CM-type for $F$, i.e.\ we choose $\sigma_1,\ldots, \sigma_d$ embeddings $F \hookrightarrow \C$ such that $\sigma_{i|F_0} = \tau_i$. In particular, $\Hom(F,\C)= \set{\sigma_i, \bar \sigma_i}_i$.

We fix $\mc O_B$, an order of $B$ that is stable under $^\ast$. Let $(\Lambda, \langle \cdot, \cdot \rangle, h)$ be a PEL type $\mc O_B$-lattice in the sense of \cite[Definition~1.2.1.3]{lan}, with $\Lambda \neq 0$. We set $V \colonequals \Lambda \otimes_{\Z} \Q$. We obtain an algebraic group $G$ over $\Sp(\Z)$ as in \cite[1.2.1.6]{lan}. Thanks to our assumptions, $G_{\Q}$ is a reductive connected algebraic group over $\Q$.
\begin{ass}
We assume that the complex dimension of the symmetric space associated to $G$ is at least $2$ (so Koecher's principle holds).
\end{ass}
We decompose
\begin{gather} \label{eq: dec C}
V \otimes_{\Q} \C \cong V_{\C,1} \oplus V_{\C,2}
\end{gather}
in such a way that $h(z)$ acts on $V_{\C,1}$ via multiplication by $z$ and via multiplication by $\bar z$ on $V_{\C,2}$. We write $E$ for the reflex field, the finite extension of $\Q$ defined as the field of definition of the isomorphism class of the complex $B$-representation $V_{\C,1}$.

We let $p \neq 2$ be a prime number, fixed from now on. We assume there is an isomorphism, that we fix,
\begin{gather} \label{eq: dec int}
\mc O_{B,p} \cong \prod_{\mathfrak p | p} \mat_n(\mc O_{\mathfrak p}).
\end{gather}
where the product is over the prime ideals of $\mc O_F$ above $p$ and $\mc O_{\mathfrak p}$ is a finite extension of $\Z_p$. We choose a uniformizer element $\varpi_{\mathfrak p} \in \mc O_{\mathfrak p}$. We assume that $\mc O_{B,p}$ is a maximal order of $B_p$ and that the restriction of $\langle \cdot, \cdot \rangle$ to $\Lambda_p$ gives a perfect pairing with values in $\Z_p$. By \eqref{eq: dec int}, we have decompositions
\[
V_p \cong \prod_{\mathfrak p | p} V_{\mathfrak p} \mbox{ and } \Lambda_p \cong \prod_{\mathfrak p | p} \Lambda_{\mathfrak p}.
\]
Taking multiples by powers of $\varpi_{\mathfrak p}$ of $\Lambda_{\mathfrak p}$, we obtain a selfdual chain $\mc L_{\mathfrak p}$ of $\mat_n(\mc O_{\mathfrak p})$-lattices of $V_{\mathfrak p}$. The product of the $\mc L_{\mathfrak p}$ gives $\mc L_p$, a selfdual multichain of $\mc O_{B,p}$-lattices in $V_p$ (see \cite[Chapter~3]{rapoport_zink} for the definition of these notions). Let $K_p \subset G(\Q_p)$ be the stabilizer of $\mc L_p$. Then $K_p$ is a parahoric subgroup. We fix $\mc H \subset G({\widehat{\Z}}^p)$, a compact open subgroup that we assume to be \emph{neat} (see \cite[Definition~1.4.1.8]{lan}). We will denote with $N$ a positive integer not divisible by $p$ such that $\mc U^p(N) \subset \mc H$ (see \cite[Remark~1.2.1.9]{lan} for the definition of $\mc U^p(N)$).
\begin{rmk} \label{rmk: unr case}
We have imposed the condition that the multichain $\mc L_p$ comes from a single lattice $\Lambda_p$ in such a way that, if $B$ is unramified at $p$, we are in the situation of \cite{kottwitz} and \cite{lan}. In this case $G$ is unramified over $\Q_p$ and $K_p$ is an hyperspecial subgroup.
\end{rmk}
We fix once and for all embeddings $\overline \Q \hookrightarrow \C$ and $i_p \colon\overline \Q \hookrightarrow \overline \Q_p$. We denote with $\mc P$ the corresponding prime ideal of $\mc O_E$ above $p$ and we write $E_{\mc P}$ for the $\mc P$-adic completion of $E$. We are interested in the functor
\[
Y \colon \mbox{ locally noetherian } \mc O_{E_{\mc P}}-\mbox{schemes} \to \mathbf{set}
\]
that to $S$ associates the isomorphism classes of the following data:
\begin{enumerate}
 \item an abelian scheme $A/S$;
 \item a polarization $\lambda \colon A \to A^\vee$ of degree prime to $p$;
 \item an action of $\mc O_B$ on $A/S$ as in \cite[Definition~1.3.3.1]{lan};
 \item a $\mc H$-level structure in the sense of \cite[Definition~1.3.7.6]{lan}.
\end{enumerate}
We furthermore require the usual determinant condition of Kottwitz, see \cite[Definition~1.3.4.1]{lan}.
\begin{rmk} \label{rmk: is clas iso class}
We have defined $Y$ using isomorphism classes of abelian schemes, rather than isogeny classes as is done in \cite{rapoport_zink} and \cite{kottwitz}. By \cite[Proposition~1.4.3.4]{lan}, these two approaches are equivalent.
\end{rmk}
\begin{teono}[{\cite[\S~5]{kottwitz}, \cite[\S~6.9]{rapoport_zink}, and \cite[Theorem~1.4.1.11 and Corollary~1.4.1.12]{lan}}] The functor $Y$ is representable by a quasi-projective scheme over $\Sp(\mc O_{E_{\mc P}})$, denoted again by $Y$. If $B$ is unramified at $p$, then $Y$ is smooth over $\Sp(\mc O_{E_{\mc P}})$.
\end{teono}
\begin{ass} \label{ass: ord}
We assume that the ordinary locus of the reduction modulo $\mc P$ of $Y$ is Zariski dense.
\end{ass}
\begin{rmk} \label{rmk: ord}
If $B$ is unramified at $p$, by \cite[1.6.3]{ord_pel}, the above assumption is equivalent to the fact that $E_{\mc P}$ is isomorphic to $\Q_p$ and it is automatically satisfied in case (C).
\end{rmk}
Let $\tilde K$ be a number field such that the decomposition in \eqref{eq: dec C} is defined over $\tilde K$ and let $K$ be the completion of $\tilde K$ at the prime ideal above $p$ given by our fixed embedding $i_p \colon \bar \Q \hookrightarrow \bar \Q_p$. It is a finite extension of $\Q_p$ and we choose a uniformizer element $\varpi$. We freely enlarge $K$ without any comment. We have decompositions of $\mc O_B$-modules
\[
V \otimes_{\Q} K \cong V_1 \oplus V_2 \mbox{ and } \Lambda \otimes_{\Q_p} K \cong \Lambda_1 \oplus \Lambda_2,
\]
where $\Lambda_i$ is a $\mc O_K$-lattice in $V_i$. We base change $Y$ to $\mc O_K$, using the same notation. As shown in \cite{pappas_non_flat}, $Y$ can be not flat over $\mc O_K$. We are interested in admissible formal schemes that are integrally closed in its generic fiber. Starting with $Y$, we perform the following steps to obtain such a formal scheme:
\begin{itemize}
 \item let $\widetilde Y$ be the flat closure of $Y$ in $Y_K$;
 \item let $\widetilde{\mathfrak Y}$ be the $\varpi$-adic completion of $\widetilde Y$. It is an admissible formal scheme over $\Spf(\mc O_K)$;
 \item let $\mathfrak Y$ be the normalization of $\widetilde{\mathfrak Y}$ in its generic fiber.
\end{itemize}
In this way $\mathfrak Y$ is an admissible formal scheme and we have its generic fiber $\mathfrak Y^{\rig}$. We follow the notation introduced in \cite[\S~4.1]{AIP}, in particular $\Nadm$ is the category of admissible $\mc O_K$-algebras $R$ that are integrally closed in $R[1/p]$. We will freely use the fact that all our formal schemes have a nice moduli interpretation when restricted to objects of $\Nadm$ (see \cite[Proposition~5.2.1.1]{AIP}). In particular, if the canonical subgroup exists over $R[1/p]$ (see below), it automatically extends to $R$, see \cite[Proposition~4.1.3]{AIP}.
\subsection{The Hasse invariant and the canonical subgroup} \label{subsec: can sub} Let $(p) = \prod_{i=1}^k (\varpi_i)^{e_i}$ be the decomposition of $(p)$ in $\mc O_{F_0}$ and let $\mc O_i$ be the completion of $\mc O_{F_0}$ with respect to $(\varpi_i)$ (here $\varpi_i$ is a fixed uniformizer of $\mc O_i$). We have a decomposition $\Hom(F_0,\C_p) = \coprod_{i=1}^k D_i$, where $D_i$ is the set of the embeddings $F_0 \hookrightarrow \C_p$ coming from $(\varpi_i)$. We have $\mc O_{F_0,p} \cong \prod_{i=1}^k \mc O_i$. We set $d_i \colonequals [F_i:\Q_p]$, where $F_i\colonequals \Fr(\mc O_i)$, so we have $|D_i|=d_i$. We write $d_i=e_if_i$. From now on, we assume that $K$ is big enough to contain the image of all embeddings $F \hookrightarrow \C_p$. In this section $A$ will be an abelian scheme given by the moduli problem associated to $Y$. We assume that $A$ is defined over a finite extension of $\mc O_K$, so it comes from a rigid point of $\mathfrak Y^{\rig}$.
\subsection*{Case (A)}
Let $B$ be of type (A). For any $i=1,\ldots,d$, we have the $B \otimes_{F_0,\tau_i} \mathds{R} \cong \mat_n(\C)$-module $V \otimes_{F_0,\tau_i} \mathds{R}$. We can write $V \otimes_{F_0,\tau_i} \C \cong \C^n \otimes_{\C} W_i$ for an essentially unique $\C$-vector space $W_i$. Moreover, $W_i$ naturally inherits an hermitian form from $V$. We write $(a_i^+,a_i^-)$ for its signature. We have $a_i^+ + a_i^-=\frac{\dim_{\Q}(V)}{2nd}$ for all $i$.
\begin{ass} \label{ass: ord type A}
We assume that each $(\varpi_i)$ splits completely in $\mc O_F$, and we write $(\varpi^\pm_i)$ for the prime ideals of $\mc O_F$ above $(\varpi_i)$. Moreover, if $i_1,i_2 = 1,\ldots,d$ are such that $i_p \circ \sigma_{i_1}$ and $i_p \circ \sigma_{i_2}$ define the same $p$-adic valuation, we assume $a_{i_1}^+=a_{i_2}^+$.
\end{ass}
\begin{rmk}
If $p$ is unramified in $\mc O_B$ and each $(\varpi_i)$ splits in $\mc O_F$, then the above condition on the signature is equivalent to Assumption~\ref{ass: ord}.
\end{rmk}
\begin{rmk} \label{rmk: split not nec}
The assumption that each $(\varpi_i)$ splits in $\mc O_F$ is not necessary. If $(\varpi_i)$ is inert the theory is similar to case (C). We leave the details to the interested reader.
\end{rmk}
Using the obvious notation, we can rewrite the decomposition in \eqref{eq: dec int} as
\begin{gather} \label{eq: dec case A}
\mc O_{B,p} = \prod_{i=1}^k \left( \mat_n(\mc O_i^+) \oplus \mat_n(\mc O_i^-) \right ).
\end{gather}
We can assume that the left ideal of $\mc O_{B,p}$ generated by $\varpi_i^+$ corresponds to the left ideal generated by the $2k$ matrices $M_j^\pm$, where $M_i^+ = \diag(\varpi_i^+,\ldots,\varpi_i^+)$ and $M_j^\pm = 1$ otherwise. We have a decomposition
\begin{gather*}
A[p^\infty] = \prod_{i=1}^k \left ( A[((\varpi_i^+)^{e_i})^\infty] \oplus A[((\varpi_i^-)^{e_i})^\infty] \right ),
\end{gather*}
where $A[((\varpi_i^-)^{e_i})^\infty]$ is canonically identified with the Cartier dual of $A[((\varpi_i^+)^{e_i})^\infty]$. Using the canonical isomorphisms $\mc O_i \cong \mc O_i^+ \cong \mc O_i^-$, we will consider only $\mc O_i$.
\subsection*{Case (C)} Let $B$ be of type (C). Similarly to case (A), we can write the $B \otimes_{F,\tau_i} \mathds{R} \cong \mat_n(\mathds{R})$-module $V \otimes_{F,\tau_i} \mathds{R}$ as $\mathds{R}^n \otimes_{\mathds{R}} W_i$ for an essentially unique $\mathds{R}$-vector space $W_i$. We have $a_i \colonequals \dim_{\mathds{R}} W_i = \frac{\dim_{\Q}(V)}{2nd}$. We can rewrite the decomposition in \eqref{eq: dec int} as
\begin{gather} \label{eq: dec case C}
\mc O_{B,p} = \prod_{i=1}^k \mat_n(\mc O_i).
\end{gather}
We can assume that, under the isomorphism \eqref{eq: dec case C}, the left ideal of $\mc O_{B,p}$ generated by $\varpi_i$ corresponds to the left ideal generated by the $k$ matrices $M_j$, where $M_i = \diag(\varpi_i,\ldots,\varpi_i)$ and $M_j = 1$ otherwise. We have a decomposition
\begin{gather*}
A[p^\infty] = \prod_{i=1}^k A[(\varpi_i^{e_i})^\infty],
\end{gather*}
where $A[(\varpi_i^{e_i})^\infty]$ is endowed with a principal $\mat_n(\mc O_i)$-linear polarization.

If $G$ is a Barsotti-Tate group defined over a finite extension of $\Z_p$, we write, as in \cite{AIP}, $\Hdg(G) \in [0,1]$ for the \emph{truncated} valuation of any lift of the Hasse invariant of the special fiber of $G$ (note that $\Hdg(G)$ is denoted $\Ha(G)$ in \cite{fargues_can}). We have a function
\begin{gather*}
\Hdg = (\Hdg_i)_i \colon \mathfrak Y^{\rig} \to [0,1]^k \\
A \mapsto (\Hdg(G_i^\pm))_i
\end{gather*}
where $G_i^\pm \colonequals e_{1,1} \cdot A[((\varpi_i^\pm)^{e_i})^\infty]$ (in case (A), since $G_i^-$ is the Cartier dual of $G_i^+$, we have $\Hdg(G_i^+) = \Hdg(G_i^-)$, so there is no ambiguity in the notation $\Hdg(G_i^\pm)$). If $\underline v = (v_i)_i \in [0,1]^k$ we set
\[
\mathfrak Y(\underline v)^{\rig} \colonequals \set{x \in \mathfrak Y^{\rig} \mbox{ such that } \Hdg(x)_i \leq v_i \mbox{ for all } i} .
\]
The ordinary locus of $\mathfrak Y^{\rig}$ is $\mathfrak Y(0)^{\rig}$, it coincides with the tube of the ordinary locus of the special fiber of $\mathfrak Y$. It is not empty by Assumption~\ref{ass: ord}. If $\underline v \in \Q^k \cap [0,1]^k$, we have that $\mathfrak Y(\underline v)^{\rig}$ is a quasi-compact strict neighbourhood of $\mathfrak Y(0)^{\rig}$.

We are going to define, for all $\underline v \in [0,1]^k$, a canonical formal model $\mathfrak Y(\underline v)$ of $\mathfrak Y(\underline v)^{\rig}$, following the approach of \cite[Definition~III.2.11]{peter_tors}. Let $\underline \omega_i^\pm$ be conormal sheaf of $G_i^\pm$ (below $\underline \omega_i^\pm$ will have a slightly different meaning, but no confusion should arise). The Hasse invariant defines a section, denoted $\Ha_i^\pm$, of $\det(\underline \omega_i^\pm)^{\otimes p-1}$ on the reduction modulo $p$ of $\mathfrak Y$. For all $i$, there is a canonical isomorphism $\underline \omega_i^+ \cong \underline \omega_i^-$, and the two Hasse invariants $\Ha_i^+$ and $\Ha_i^-$ are identified under the corresponding isomorphism. For this reason, we will simply write $\underline \omega_i$ and $\Ha_i$.
\begin{defi} \label{defi: form model}
Let $\underline v = (v_i)_{i=1}^k$ and assume that, for all $i$, there is in $\mc O_K$ an element, denoted $p^{v_i}$, of valuation $v_i$. For all $j=1\ldots, k$, we define $\tilde{\mathfrak Y}(v_1,\ldots, v_j)$ by recursion as the functor sending any $p$-adically complete flat $\mc O_K$-algebra $S$ to the set of equivalence classes of pairs $(f,u)$, where:
\begin{itemize}
 \item $f \colon \Spf(S) \to \tilde{\mathfrak Y}(v_1,\ldots, v_{j-1})$ (if $j=1$, we set $\tilde{\mathfrak Y}(v_1,\ldots, v_{j-1}) \colonequals \mathfrak Y$);
 \item $u \in \Homol^0(\Spf(S), \det(\underline\omega_j)^{\otimes p-1})$ is a section such that, in $S/p$, we have the equality
\[
u \Ha_j(\bar f) = p^{v_j} \in S/p,
\]
where $\bar f$ is the reduction of $f$ modulo $p$ (to be precise we should first of all consider the pullback of $\underline \omega_i$ and $\Ha_i$ via the morphism $\tilde{\mathfrak Y}(v_1,\ldots, v_{j-1}) \to \mathfrak Y$).
\end{itemize}
Two pairs $(f,u)$ and $f',u'$ are equivalent if $f=f'$ and there is some $h \in S$ such that $u' = u(1+p^{1-v_j}h)$.
\end{defi}
By \cite[Lemma~III.2.13]{peter_tors}, we have that $\tilde{\mathfrak Y}(\underline v) \colonequals \tilde{\mathfrak Y}(v_1,\ldots, v_k)$ is representable by a formal scheme, flat over $\mc O_K$. Moreover, one has the usual local description of $\tilde{\mathfrak Y}(\underline v)$. We define $\mathfrak Y(\underline v)$ as the normalization of $\tilde{\mathfrak Y}(\underline v)$ in its generic fiber.
\begin{rmk} \label{rmk: form mod}
The point of the definition of $\mathfrak Y(\underline v)$ is that, using this approach, we do not need to worry about whether the various Hasse invariants lift to characteristic $0$ (and we do not choose any such lift).
\end{rmk}
\begin{notation} \label{not: epsilon}
For any integer $n \geq 0$, we write $\varepsilon_n=\frac{1}{2p^{n-1}}$ if $p\neq 3$ and $\varepsilon_n=\frac{1}{3^n}$ if $p=3$. Unless explicitly stated, in the sequel we will always assume that $v_i < \varepsilon_n$ for all $i$, where $n$ will be clear from the context.
\end{notation}
We fix an integer $n \geq 0$. If $A/R$ is an abelian scheme above $\mathfrak Y(\underline v)$, where $R \in \Nadm$, we have, by \cite[Théorème~6]{fargues_can} and \cite[Proposition~4.1.3]{AIP}, a canonical subgroup of $H_{i,n}^\pm \subset G_i^\pm[p^n]$.
\begin{rmk} \label{rmk: con sub perp}
In case (A) we have $H_{i,n}^- = (H_{i,n}^+)^\perp$, where the orthogonal is taken with respect to the perfect pairing given by duality.
\end{rmk}
Over $K$, we fix $\mc O_i$-linear compatible isomorphisms of étale group schemes
\[
\left (\mc O_i/(\varpi_i^\pm)^{e_in}\right)^{\Dual} \cong \mc O_i/(\varpi_i^\pm)^{e_in},
\]
where $(\cdot)^{\Dual}$ denotes Cartier duality. In particular we will assume that $K$ contains the necessary roots of unity.
\begin{lemma} \label{lemma: can sub free}
We have that $H_{i,n}^\pm$ has rank $p^{na_i^\pm d_i}$ and is stable under $\mc O_i$. Moreover, locally for the étale topology on $R_K$, it is isomorphic to $(\mc O_i/p^n)^{a_i^\pm}$. The same is true for $(H_{i,n}^\pm)^{\Dual}$.
\end{lemma}
\begin{proof}
The statement about the rank follows from \cite[Théorème~6]{fargues_can} and, by \cite[Corollaire~10]{farg_harder}, we have that $H_{i,n}^\pm$ is $\mc O_i$-stable. Again \cite[Théorème~6]{fargues_can} implies that $H_{i,n}^\pm$ is, locally for the étale topology on $R_K$, isomorphic to $(\Z/p^n\Z)^{a_i^\pm d_i}$, so we can show that $H_{i,n}^\pm$ is not killed by $(\varpi_i^\pm)^{e_in-1}$. The dimension of $G_i^\pm$ is $a_i^\pm d_i$, so we have $\deg(G_i^\pm [p^n])=na_i^\pm d_i$ (see \cite[\S~3]{farg_harder} for details about the degree). Moreover, multiplication by $\varpi_i^\pm$ on $G_i^\pm$ is an isogeny, so, for all $s$ we have $\deg(G_i^\pm[(\varpi_i^\pm)^s]) = \val(\det( (\varpi_i^\pm)^{s,\ast})) = s \val(\det((\varpi_i^\pm)^\ast))$, where $(\varpi_i^\pm)^{s,\ast} \colon \underline \omega_{G_i^\pm} \to \underline \omega_{G_i^\pm}$ is the pullback. In particular we have $\deg(G_i^\pm[(\varpi_i^\pm)^{e_in-1}])=(e_in-1)a_i^\pm f_i$. By \cite[Théorème~6]{fargues_can}, we have $\deg(H_{i,n}^\pm)=na_i^\pm d_i-\frac{p^n -1}{p-1}\Ha(G_i^\pm[p^n])$, so we see that $\deg(G_i^\pm[(\varpi_i^\pm)^{e_in-1}]) < \deg(H_{i,n}^\pm)$ and we conclude by \cite[Lemme~4]{farg_harder}.
\end{proof}
\begin{notation}
We consider the algebraic group $\GL^{\mc O}$ over $\Z_p$ defined, in case (A) and (C) respectively, by
\[
\GL^{\mc O} \colonequals \prod_{i=1}^k \Res_{\mc O_i/\Z_p}(\GL_{a_i^+} \times \GL_{a_i^-}) \mbox{ and } \GL^{\mc O} \colonequals \prod_{i=1}^k \Res_{\mc O_i/\Z_p}\GL_{a_i}.
\]
We also have the subgroup $\T^{\mc O}$ defined by
\[
\T^{\mc O} \colonequals \prod_{i=1}^k \Res_{\mc O_i/\Z_p}(\m G^{a_i^+}_{\operatorname{m}} \times \m G^{a_i^-}_{\operatorname{m}}) \mbox{ and } \T^{\mc O} \colonequals \prod_{i=1}^k \Res_{\mc O_i/\Z_p}\m G^{a_i}_{\operatorname{m}}.
\]
Note that, over $K$, we have that $\T^{\mc O}$ is a split torus. We consider the Borel subgroup $\B^{\mc O}$ given by the `upper triangular matrices', with unipotent radical $\U^{\mc O}$.
\end{notation}
We now introduce the Shimura varieties we need. We write $\mathfrak Y(p^n)(\underline v)^{\rig} \to \mathfrak Y(\underline v)^{\rig}$ for the finite étale covering that, in case (A), parametrizes $\mc O_i$-linear trivializations $H_{i,n}^+ \oplus H_{i,n}^- \cong (\mc O_i/(\varpi_i^+)^{e_in})^{a_i^+} \oplus (\mc O_i/(\varpi_i^-)^{e_in})^{a_i^-}$ (note that everything is in characteristic $0$ here); in case (C) it parametrizes $\mc O_i$-linear trivializations $H_{i,n} \cong (\mc O_i/\varpi_i^{e_in})^{a_i}$. There is an action of $\GL^{\mc O}(\Z_p)$ on $\mathfrak Y(p^n)(\underline v)^{\rig}$. Let $\mathfrak Y_{\Iw}(p^n)(\underline v)^{\rig}$ be the quotient of $\mathfrak Y(p^n)(\underline v)^{\rig}$ with respect to $\B^{\mc O}(\Z_p)$. Finally, we let $\mathfrak Y_{\Iwt}(p^n)(\underline v)^{\rig}$ be the quotient of $\mathfrak Y(p^n)(\underline v)^{\rig}$ with respect to $\U^{\mc O}(\Z_p)$. Taking the normalization of $\mathfrak Y(\underline v)$ we obtain the tower of formal schemes
\[
\mathfrak Y(p^n)(\underline v) \to \mathfrak Y_{\Iwt}(p^n)(\underline v) \to \mathfrak Y_{\Iw}(p^n)(\underline v) \to \mathfrak Y(\underline v).
\]
Any of these formal schemes has a reasonable moduli space interpretation.

\section{\texorpdfstring{The sheaf $\mc F$ and modular sheaves}{The sheaf F and modular sheaves}} \label{sec: F}
\subsection{The weight space} \label{subsec: weight space} 
Our weight space is the rigid analytic variety $\mc W$ associated to the completed group algebra $\mc O_K \llbracket \T^{\mc O}(\Z_p) \rrbracket$. It satisfies
\[
\mc W(A)= \Hom_{\cont}(\T^{\mc O}(\Z_p), A^\ast)
\]
for any affinoid $K$-algebra $A$.

Accordingly to the decomposition of $\T^{\mc O}$, we have $\mc W = \prod_i (\mc W_i^+ \times \mc W_i^-)$ in case (A) and $\mc W = \prod_i \mc W_i$ in case (C). In particular, we can write $\chi = (\chi_i^\pm)_i$ for all $\chi \in \mc W(\C_p)$. Let $w_i^\pm > 0$ be a rational number such that there is an element $p^{w_i^\pm} \in \mc O_K$ of valuation $w_i^\pm$. We say that $\chi_i^\pm \in \mc W_i^\pm (\C_p)$ is \emph{$w_i^\pm$-locally analytic} if $\chi_i^\pm$ extends to an analytic character
\[
\chi_i^\pm \colon (\mc O_i^\ast(1+p^{w_i^\pm}\mc O_{\C_p}))^{a_i^\pm} \to \C_p^\ast.
\]
If $\underline w = (w_i^\pm)_i$ and $\chi = (\chi_i^\pm)_i \mc W(\C_p)$, we say that $\chi$ is \emph{$\underline w$-locally analytic} is each $\chi_i^\pm$ is $w_i^\pm$-locally analytic. Any $\chi \in \mc W(\C_p)$ is $\underline w$-locally analytic for some $\underline w$. Moreover, let $\mc U \subset \mc W$ be an affinoid associated to a $\C_p$-algebra $A$ and let $\chi_{\mc U}^{\un} = (\chi_{\mc U,i}^{\un,\pm})_i$ be its universal character. Then there is a tuple of positive rational numbers $\underline w = (w_i^\pm)_i$ such each $\chi_{\mc U,i}^{\un,\pm}$  extends to an analytic character
\[
\chi_{\mc U,i}^{\un,\pm} \colon (\mc O_i^\ast(1+p^{w_i^\pm}\mc O_{\C_p}))^{a_i^\pm} \to A^\ast.
\]
We say in this case that $\chi_{\mc U}^{\un}$ is $\underline w$-locally analytic.

Fix an integer $n \geq 1$. We have the subspace $\widetilde{\mc W}_i^\pm(n)$ given by those $\chi_i^\pm \in \mc W_i^\pm(\C_p)$ that satisfy $\chi_i^\pm(1+p^n \mc O_i) \subset 1+p\mc O_{\C_p}$. We define $\mc W_i^\pm(n)$ as the subspace of $\widetilde{\mc W}_i^\pm(n)$ given by the characters $\chi_i^\pm$ such that their restriction to $1+p^n\mc O_i$ is obtained from a $\Z_p$-linear morphism $p^n \mc O_i \to p\mc O_{\C_p}$ taking composition with the $p$-adic logarithm and with the $p$-adic exponential. If $w_i^\pm \geq 1$ is a rational number, we set $\mc W_i^\pm(w_i^\pm) \colonequals \mc W_i^\pm([w_i^\pm])$, where $[w_i^\pm]$ denotes the integer part of $w_i^\pm$.

Let $\underline w = (w_i^\pm)_i$ be a tuple of rational numbers. We set $\mc W(\underline w) \colonequals \prod_i (\mc W_i^+(w_i^+) \times \mc W_i^-(w_i^-))$ or $\mc W(\underline w) \colonequals \prod_i \mc W_i(w_i)$. By construction we have the following
\begin{prop} \label{prop: cov we sp}
Each $\mc W(\underline w)$ is affinoid and $\set{\mc W(\underline w)}_{\underline w}$ is an admissible covering of $\mc W$. Moreover, if $\chi \in \mc W(\underline w)(K)$, then $\chi$ is $\underline w$-analytic.
\end{prop}
If $\chi \in X^\ast(\T^{\mc O})$ is a character of $\T^{\mc O}$, we define $\chi' \colonequals -w_0 \chi$, where $w_0$ is the longest element of the Weyl group of $\GL^{\mc O}$ with respect to $\T^{\mc O}$. For any $\underline w$, the map $\chi \mapsto \chi'$ extends to an involution of $\mc W(\underline w)$, denoted in the same way.
\subsection{\texorpdfstring{The sheaf $\mc F$}{The sheaf F}} \label{subsec: sheaf F} Let $R$ be in $\Nadm$. Suppose we are given a morphism $\Spf(R) \to \mathfrak Y(p^n)(\underline v)$, so we have an abelian scheme $A \to \Sp(R)$. We write $e \colon \Sp(R) \to A$ for the zero section. We consider the sheaf $e^\ast \Omega^1_{A / \Sp(R)}$. It is a $\mc O_B \otimes_{\Z} R$-module and a locally free sheaf of $\mc O_{\Sp(R)}$-modules of rank $\dim_{\Q}(V)/2$. By Morita's equivalence, $e^\ast \Omega^1_{A / \Sp(R)}$ corresponds to a sheaf $\underline \omega$ and we can write, in case (A) and (C) respectively,
\[
\underline \omega = \bigoplus_{i=1}^k (\underline \omega_i^+ \oplus \underline \omega_i^-) \mbox{ and } \underline \omega = \bigoplus_{i=1}^k \underline \omega_i.
\]
Let $i$ and $n$ be fixed. The map (see \cite[\S~4.2]{AIP})
\[
\HT_{(H_{i,n}^\pm)^{\Dual}} \colon (H_{i,n}^\pm)^{\Dual}(R_K) \to \underline \omega_{H_{i,n}^\pm}
\]
respects the action of $\mc O_i$ by functoriality. We define $\mc F_i^\pm(R)$ as the sub $\mc O_i \otimes_{\Z_p} R$-module of $\underline \omega_i^\pm$ generated by the inverse image of $\HT_{(H_{i,n}^\pm)^{\Dual}}(R_K)$ under the natural map, given by pullback, $\underline \omega_i^\pm \to \underline \omega_{H_{i,n}^\pm}$. We have that $\mc F_i^\pm$ does not depend on $n$.
\begin{prop} \label{prop: F loc free}
The sheaf $\mc F_i^\pm \subset \underline \omega_i^\pm$ is a locally free sheaf of $\mc O_i \otimes_{\Z} \mc O_{\Spf(R)}$-modules that contains $p^{\frac{v_i}{p-1}}\underline \omega_i^\pm$. If $w_i^\pm \in ]0,n-v_i\frac{p^n}{p-1}]$ then we have a natural map $\HT_{i,w_i^\pm}^\pm \colon (H_{i,n}^\pm)^{\Dual} (R_K) \to \mc F_i^\pm(R) \otimes_R R_{w_i^\pm}$ such that the induced map
\begin{gather} \label{eq: iso can sub}
(H_{i,n}^\pm)^{\Dual}(R_K) \otimes_{\Z_p} R \to \mc F_i^\pm(R) \otimes_R R_{w_i^\pm}
\end{gather}
is an isomorphism of $\mc O_i \otimes_{\Z_p} R$-modules.
\end{prop}
\begin{proof}
Taking into account Lemma~\ref{lemma: can sub free}, the proof is similar to the one of \cite[Proposition~4.3.1]{AIP}.
\end{proof}
We define
\[
\mc F \colonequals \bigoplus_{i=1}^k (\mc F_i^+ \oplus \mc F_i^-) \mbox{ or } \mc F \colonequals \bigoplus_{i=1}^k \mc F_i,
\]
by Morita's equivalence and Proposition~\ref{prop: F loc free}, it corresponds to a locally free sheaf of $\mc O_B \otimes_{\Z} \mc O_{\Sp(R)}$-modules contained in $e^\ast \Omega^1_{A/\Sp(R)}$. Moreover, this inclusion becomes an isomorphism if we invert $p$.
\subsection{Modular sheaves} \label{subsec: mod sheaves}
Let $n\geq 1$ be an integer and let $w_i^\pm \in ]0, n-v_i\frac{p^n}{p-1}]$ be a rational number. We begin with case (A). For each $i$, there are formal schemes $\Iwtform_{i,w_i^\pm}^\pm \to \mathfrak Y(p^n)(\underline v)$ defined as follows. Let $R$ be in $\Nadm$ and suppose that $\mc F_i^\pm(R)$ is free. The $R$-points of $\Iwtform_{i,w_i^\pm}^\pm$ correspond naturally to the following data:
\begin{itemize}
 \item an $R$-point of $\mathfrak Y(p^n)(\underline v)$,
 \item a filtration
\[
\Fil_\bullet \mc F_i^\pm(R)=(0=\Fil_0 \mc F_i^\pm(R) \subset \cdots \subset \Fil_{a_i^\pm}\mc F_i^\pm(R) = \mc F_i^\pm(R)); 
\]
 \item trivializations
\[
\Gr_j \Fil_\bullet \mc F_i^\pm(R) \colonequals \Fil_j \mc F_i^\pm(R)/\Fil_{j-1} \mc F_i^\pm(R) \cong (\mc O_i \otimes_{\Z} R)^j
\]
\end{itemize}
such that the following conditions hold:
\begin{itemize}
 \item $\Fil_j \mc F_i^\pm(R)$ is a free $\mc O_i \otimes_{\Z} R$-module for each $0 \leq j \leq a_i^\pm$;
 \item $\Fil_\bullet \mc F_i^\pm(R)$ corresponds, modulo $p^{w_i^\pm} R$ and via the isomorphism in \eqref{eq: iso can sub}, to the filtration on $(H_{i,n}^\pm)^{\Dual}(R_K) \otimes_{\Z} R$ given by the trivialization of $(H_{i,n}^\pm)^{\Dual}$;
 \item the trivializations of $\Gr_j \Fil_\bullet \mc F_i^\pm(R)$ are compatible, modulo $p^{w_i^\pm} R$ and via the isomorphism in \eqref{eq: iso can sub}, to the the trivializations of $(H_{i,n}^\pm)^{\Dual}$.
\end{itemize}
We set $\Iwtform_{\underline w} \colonequals \prod_{i=1}^k (\Iwtform_{i,w_i^+}^+ \times \Iwtform_{i,w_i^-}^-)$. We leave to the reader the definition of $\Iwtform_{\underline w} \to \mathfrak Y(p^n)(\underline v)$ in case (C). Let $v_i'\leq v_i$ for all $i$. We can repeat the above definition to obtain a formal scheme $\Iwtform'_{\underline w} \to \mathfrak Y(p^n)(\underline v')$. The restriction of $\Iwtform_{\underline w}$ to $\mathfrak Y(p^n)(\underline v')$ is naturally isomorphic to $\Iwtform_{\underline w}'$, so we can safely omit $\underline v$ from the notation $\Iwtform_{\underline w}$.

With the obvious notation, we define a formal group $\mathfrak T_{\underline w}^{\mc O}$ by
\[
\mathfrak T_{\underline w}^{\mc O}(R) \colonequals \ker(\T^{\mc O}(R) \to \T^{\mc O}(R_{\underline w})),
\]
and we make a similar definition for $\mathfrak B_{\underline w}^{\mc O}$ and $\mathfrak U_{\underline w}^{\mc O}$. We write $\T_{\underline w}^{\mc O}$, $\B_{\underline w}^{\mc O}$, and $\U_{\underline w}^{\mc O}$ for the corresponding rigid fibers. We have a natural action of $\B^{\mc O}(\Z_p) \mathfrak B_{\underline w}^{\mc O}$ on $\Iwtform_{\underline w}$ over $\mathfrak Y_{\Iw}(p^n)(\underline v)$.

Let $\chi \in \mc W(\underline w)(K)$ be a character. We set $\chi'(\U^{\mc O}(\Z_p) \mathfrak U_{\underline w}^{\mc O}(\mc O_{\C_p})) = 1$. Since $\chi$ is $\underline w$-locally analytic by Proposition~\ref{prop: cov we sp}, we can extend $\chi'$ to an analytic character
\[
\chi' \colon \B^{\mc O}(\Z_p) \mathfrak B_{\underline w}^{\mc O}(\mc O_{\C_p}) \to \C_p^\ast.
\]
We consider the morphism, obtained by composition,
\[
\pi \colon \Iwtform_{\underline w} \to \mathfrak Y_{\Iw}(p)(\underline v).
\]
\begin{defi} \label{defi: sheaf form}
We define the sheaf
\[
\underline{\mathfrak w}^{\dagger \chi}_{\underline v,\underline w} \colonequals \pi_\ast \mc O_{\Iwtform_{\underline w}}[\chi'],
\]
where $[\chi']$ means that we consider the subspace of homogeneous sections of degree $\chi'$ for the action of $\B^{\mc O}(\Z_p)\mathfrak B_{\underline w}$. We call $\underline{\mathfrak w}^{\dagger \chi}_{\underline v,\underline w}$ the $\underline v$-overconvergent, $\underline w$-analytic integral modular sheaf of weight $\chi$.
\end{defi}
\begin{prop} \label{prop: ban sheaf}
We have that $\underline{\mathfrak w}^{\dagger \chi}_{\underline v,\underline w}$ is a formal Banach sheaf (see the appendix of \cite{AIP}).
\end{prop}
\begin{proof}
This is proved in exactly the same way as \cite[Proposition~5.2.2.2]{AIP}.
\end{proof}
The rigid fiber of $\underline{\mathfrak w}^{\dagger \chi}_{\underline v,\underline w}$ is denoted $\underline \omega_{\underline v,\underline w}^{\dagger\chi}$. By definition it is the $\underline v$-overconvergent, $\underline w$-analytic modular sheaf of weight $\chi$.

We can define $\underline \omega_{\underline v,\underline w}^{\dagger\chi}$ directly as follows. Since $w_i^\pm < n$ for all $i$, the natural action of $\U^{\mc O}(\Z_p)$ on $\mathfrak Y(p^n)(\underline v)^{\rig}$ induces an action of $\U^{\mc O}(\Z_p)$ on $\Iwtform_{\underline w}^{\rig}$. Taking the quotient, we obtain a rigid space $\Iwtform_{\underline w}^{\rig,\Diamond} \to \mathfrak Y_{\Iwt}(p^n)(\underline v)^{\rig}$. We have an action of $\T^{\mc O}(\Z_p)\T_{\underline w}^{\mc O}$ on $\Iwtform_{\underline w}^{\rig,\Diamond}$ over $\mathfrak Y_{\Iw}(p)(\underline v)^{\rig}$ and there is an equality
\[
\underline \omega_{\underline v,\underline w}^{\dagger\chi} = \pi_\ast^{\Diamond} \mc O_{\Iwtform_{\underline w}^{\rig,\Diamond}}[\chi'],
\]
where $\pi^{\Diamond} \colon \Iwtform_{\underline w}^{\rig,\Diamond} \to \mathfrak Y_{\Iw}(p)(\underline v)^{\rig}$ is the natural morphism and we take homogeneous sections for the action of $\T^{\mc O}(\Z_p)\T_{\underline w}^{\mc O}$.

\section{Modular forms} \label{sec: mod forms}
Since we do not have a Koecher principle for sections of our modular sheaves, to define modular forms we find it convenient to work with the compactified variety. In particular we need the following
\begin{ass}
From now on we assume that $p$ is unramified in $\mc O_B$, in particular $p$ is a good prime in the sense of \cite[Definition~1.4.1.1]{lan}. We also assume \cite[Condition~1.4.3.10]{lan}, namely that our $\mc O_B$ lattice $(\Lambda, \langle \cdot, \cdot \rangle, h)$ is such that the action of $\mc O_B$ extends to an action of some maximal order $\mc O_B' \supset \mc O_B$. This is not a restriction, see \cite[Remark~1.4.3.9]{lan}. 
\end{ass}
We fix once and for all a compatible choice of admissible smooth rational polyhedral cone decomposition data for $Y$, as in \cite[Definition~6.3.3.4]{lan}. Associated to this choice there is an arithmetic toroidal compactification $Y^{\tor}$ of $Y$, see \cite[Theorem~6.4.1.1]{lan} for the main properties of $Y^{\tor}$. Let $\Sp(R_{\alg})$ be part of the data giving a good algebraic model for some representative of a cusp label associated to $Y^{\tor}$, as in \cite[Definition~6.3.2.5]{lan} and let $R \in \Nadm$ be the $p$-adic completion of $R_{\alg}$. We set $S \colonequals \Sp(R)$, so we have a semiabelian scheme $A \to S$. Let $U \subseteq S$ be the open subset corresponding to the unique open stratum of $\Sp(R_{\alg})$, in particular $A$ is abelian over $U$. We also have a Mumford $1$-motive $M$ over $U \hookrightarrow S$ whose semiabelian part will be denoted $\tilde A \to S$. By definition $\tilde A$ is a semiabelian scheme with constant toric rank and we have $\tilde A[p^n] \hookrightarrow A[p^n]$. Here $\tilde A[p^n]$ is finite and flat, while in general $A[p^n]$ is not. As explained in \cite[Section~2.3]{ben_these}, the approximation process needed to construct good formal models can be performed in such a way that there is an isomorphism $M[p^n] \cong A[p^n]$ and we always assume that this is true. There is an action of $\mc O_B$ on $A[p^n]$, $\tilde A[p^n]$, and $M[p^n]$. The two arrows $\tilde A[p^n] \hookrightarrow A[p^n]$ and $M[p^n] \cong A[p^n]$ can be assumed to be compatible with this action.

We can now repeat the definitions of Subsection~\ref{subsec: can sub} replacing $A$ by $\tilde A$, obtaining, for all $\underline v \in [0,1]^k$, the rigid space $\mathfrak Y(\underline v)^{\tor,\rig}$ and its formal model $\mathfrak Y(\underline v)^{\tor}$.
\subsection{Modular forms} \label{subsec: mod forms}
At the end of Section~\ref{sec: PEL data}, we have introduced the rigid Shimura variety $\mathfrak Y(p^n)(\underline v)^{\rig}$ and its formal model $\mathfrak Y(p^n)(\underline v)$. We have a canonical subgroup over $\mathfrak Y(\underline v)^{\tor}$ (see for example \cite[Sections~3.3 and 4.1]{AIP}), so we can define $\mathfrak Y(p^n)(\underline v)^{\tor,\rig}$ and its formal model $\mathfrak Y(p^n)(\underline v)^{\tor}$. Using the semiabelian variety over $\mathfrak Y(p^n)(\underline v)^{\tor}$ we see that the sheaf $\mc F$ extends to $\mathfrak Y(p^n)(\underline v)^{\tor}$. The analogue of Proposition~\ref{prop: F loc free} still holds, so we can define a space
\[
\zeta \colon \Iwtform_{\underline w} \to \mathfrak Y(p^n)(\underline v)^{\tor},
\]
where $\underline w$ is as above. There is an action of $\GL^{\mc O}(\Z_p)$ on $\mathfrak Y(p^n)(\underline v)^{\tor,\rig}$ and repeating the above definitions we obtain the tower
\[
\mathfrak Y(p^n)(\underline v)^{\tor} \to \mathfrak Y_{\Iwt}(p^n)(\underline v)^{\tor} \to \mathfrak Y_{\Iw}(p^n)(\underline v)^{\tor} \to \mathfrak Y(\underline v)^{\tor}.
\]
We are interested in the morphism
\[
\pi \colon \Iwtform_{\underline w} \to \mathfrak Y_{\Iw}(p)(\underline v)^{\tor}.
\]
Repeating the above definitions, if $\chi \in \mc W(\underline w)(K)$ is a character we then have the sheaves
\[
\underline{\mathfrak w}^{\dagger \chi}_{\underline v,\underline w} \colonequals \pi_\ast \mc O_{\Iwtform_{\underline w}}[\chi'] \mbox{ and } \underline{\omega}^{\dagger \chi}_{\underline v,\underline w} \colonequals (\underline{\mathfrak w}^{\dagger \chi}_{\underline v,\underline w})^{\rig}
\]
on $\mathfrak Y_{\Iw}(p)(\underline v)^{\tor}$ and $\mathfrak Y_{\Iw}(p)(\underline v)^{\tor,\rig}$.

Let $Y_K^{\an}$ and $Y_K^{\tor,\an}$ be the analytifications of $Y_K$ and $Y_K^{\tor}$ respectively. Since $Y^{\tor}$ is proper we have a natural isomorphism $\mathfrak{Y}^{\tor,\rig} \cong Y_K^{\tor,\an}$ and in particular there is an open immersion
\[
Y_K^{\an} \hookrightarrow \mathfrak{Y}^{\tor,\rig}.
\]
We have that $Y_K^{\an}$ is dense in $\mathfrak{Y}^{\tor,\rig}$.
\begin{defi} \label{defi: over loc an}
We define the space of $\underline v$-overconvergent, $\underline w$-analytic modular forms of weight $\chi$ by
\[
\M^{\dagger \chi}_{\underline v,\underline w} \colonequals \Homol^0(\mathfrak Y_{\Iw}(p)(\underline v)^{\tor,\rig},\underline \omega_{\underline v,\underline w}^{\dagger\chi}).
\]
We define the space of overconvergent locally analytic modular forms of weight $\chi$ by
\[
\M^{\dagger \chi} \colonequals \lim_{\substack{\underline v \to  0\\ \underline w \to \infty}} \M^{\dagger \chi}_{\underline v,\underline w},
\]
where the limit is over $\underline v$ and $\underline w$ for which there is $n$ such that $\chi \in \mc W(\underline w)(K)$ and the above properties are satisfied.
\end{defi}
\begin{defi} \label{defi: bounded}
Let $\chi \in \mc W(K)$ be any continuous character. If $F$ is a global section of $\underline \omega_{\underline v,\underline w}^{\dagger\chi}$ on $\mathfrak Y_{\Iw}(p)(\underline v)^{\tor,\rig}$ we say that $F$ is bounded if it is bounded as function on $\Iwtform_{\underline w}^{\rig} \times_{\mathfrak Y^{\tor,\rig}} Y_K^{\an}$.
\end{defi}
\begin{prop} \label{prop: mod form well def}
The natural restriction morphism
\[
\M^{\dagger\chi}_{\underline v, \underline w} \to \Homol^0_{\bo}(\mathfrak Y_{\Iw}(p)(\underline v)^{\tor,\rig} \times_{\mathfrak Y^{\tor,\rig}} Y_K^{\an},\underline{\omega}^{\dagger \chi}_{\underline v,\underline w}),
\]
where $\Homol^0_{\bo}(-)$ means that we consider only \emph{bounded} sections, is an isomorphism. In particular our definition of modular forms does not depend on the choice of the toroidal compactification.
\end{prop}
\begin{proof}
The complement of $\Iwtform_{\underline w}^{\rig} \times_{\mathfrak Y^{\tor,\rig}} Y_K^{\an}$ in $\Iwtform_{\underline w}^{\rig}$ is a Zariski closed subset of codimension bigger or equal than $1$. By \cite[Theorem~1.6]{extension} any bounded function $F$ on $\Iwtform_{\underline w}^{\rig} \times_{\mathfrak Y^{\tor,\rig}} Y_K^{\an}$ extends (uniquely) to a function on $\Iwtform_{\underline w}^{\rig}$. This extension has the same weight as $F$ and gives an element of $\M^{\dagger\chi}_{\underline v, \underline w}$ as required.
\end{proof}
\begin{rmk} \label{rmk: no koec}
The reason why we do not need any Koecher principle in the proof of the above proposition is that we have defined overconvergent modular forms as section over a strict neighbourhood of the ordinary locus of $\mathfrak Y_{\Iw}(p)(\overline v)^{\tor,\rig}$, that contains also abelian varieties of bad reduction.
\end{rmk}
\subsection{\texorpdfstring{Classical modular forms}{Classical modular forms}} \label{subsec: class mod form} Fix $n$, $\underline v$ and $\underline w$ as above and assume moreover that $w_i^\pm > \frac{v_i}{p-1}$. We write $Y_{\Iw}(p)$ for the Shimura variety defined with the same PEL data as $Y$, but with Iwahoric level structure at $p$. In case (C), we have that $Y_{\Iw}(p)$ parametrizes couples $(A, (\Fil_\bullet G_i[p])_i )$, where
\begin{itemize}
 \item $A$ is an object parametrized by $Y$;
 \item $\Fil_\bullet G_i[p] = (0=\Fil_0 G_i[p] \subset \cdots \subset \Fil_{a_i} G_i[p])$ is a filtration of $G_i[p]$ made by $\mc O_i$-stable finite flat subgroups such that $\Fil_j G_i[p]$ has rank $p^{d_ij}$ and $\Fil_{a_i} G_i[p]$ is totally isotropic.
\end{itemize}
In case (A) we consider complete filtrations of $G_i^+[p]$. We consider the sheaves $\underline \omega_i^\pm$ and $\underline \omega$ on $Y_{\Iw}(p)$ defined similarly to the above ones. The unramifiedness assumption implies that $\underline \omega_i^\pm$ is a locally free $\mc O_i \otimes_{\Z_p} \mc O_{Y_{\Iw}(p)}$-module. We obtain the sheaf of trivializations of $\underline \omega$
\[
f \colon \mc T \to Y_{\Iw}(p).
\]
We let $\GL^{\mc O}$ act on $\mc T$ by $g \cdot \omega = \omega \circ g^{-1}$, where $g$ is a section of $\GL^{\mc O}$ and $\omega$ a trivialization of $\underline \omega$. In this way $\mc T$ becomes a $\GL^{\mc O}_{Y_{\Iw}(p)}$-torsor.

We write $\mathfrak Y_{\Iw}(p)$ for the formal completion of $Y_{\Iw}(p)$ along its special fiber. We have open immersions $\mathfrak Y_{\Iw}(p)(\underline v)^{\rig} \hookrightarrow \mathfrak Y_{\Iw}(p)^{\rig} \hookrightarrow Y_{\Iw}(p)_K^{\an}$. By the main result of \cite{ben_these} we have a toroidal compactification $Y_{\Iw}(p)^{\tor}$ of $Y_{\Iw}(p)$ (that we can choose in a way that is compatible with the choice we made for $Y^{\tor}$) and everything we just said extends to $Y_{\Iw}(p)^{\tor}$. Let $\mathfrak Y_{\Iw}(p)^{\tor}$ be the formal completion of $Y_{\Iw}(p)^{\tor}$ along its special fiber. The rigid space $\mathfrak Y_{\Iw}(p)(\underline v)^{\tor,\rig}$ is an open subspace of $\mathfrak Y_{\Iw}(p)^{\tor,\rig}$, but the special fiber of $\mathfrak Y_{\Iw}(p)^{\tor}$ is not the same as the special fiber of the space $\mathfrak Y_{\Iw}(p)(\underline v)^{\tor}$ defined above. In the sequel we will work only with the generic fiber of $\mathfrak Y_{\Iw}(p)^{\tor}$, so this is not a problem.

Let $X^\ast(\T^{\mc O})^+$ be the cone of dominant weights with respect to $\B^{\mc O}$. This cone is stable under $\chi \mapsto \chi'$. Let $\chi \in X^\ast(\T^{\mc O})$ be a weight. Recall that $d_i= [\mc O_i^\pm : \Q_p]$. In case (A) we can identify $\chi$ with a tuple of integers in
\[
\prod_{i=1}^k \prod_{s=1}^{d_i} ( \Z^{a_i^+}  \times \Z^{a_i^-}).
\]
We have that $\chi = (k_{i,s,j}^\pm)$ is dominant if and only if, for each $i=i,\ldots,k$ and each $s=1,\ldots,d_i$, we have
\[
k_{i,s,1}^+ \geq k_{i,s,2}^+ \geq \cdots \geq k_{i,s,a_i^+}^+ \mbox{ and } k_{i,s,1}^- \geq k_{i,s,2}^- \geq \cdots \geq k_{i,s,a_i^-}^-.
\]
In case (C) we can identify $\chi$ with a tuple of integers in
\[
\prod_{i=1}^k \prod_{s=1}^{d_i} \Z^{a_i}.
\]
We have that $\chi = (k_{i,s,j})$ is dominant if and only if, for each $i=i,\ldots,k$ and each $s=1,\ldots,d_i$ we have
\[
k_{i,s,1} \geq k_{i,s,2} \geq \cdots \geq k_{i,s,a_i}.
\]
We have that $\mathcal T$ extends to the toroidal compactification and if $\chi$ is a dominant weight then the space of classical modular forms of weight $\chi$ and Iwahoric level is by definition
\[
\M^\chi \colonequals \Homol^0(\mathfrak Y_{\Iw}(p)^{\tor,\rig},\underline \omega^\chi),
\]
where $\underline \omega^\chi$ is the subsheaf of $f_\ast \mc T$ given by homogeneous sections, for the action of $\B^{\mc O}_{Y_{\Iw}(p)}$, of degree $\chi'$. Note that the action of $\GL^{\mc O}$ on $\mc T$ induces an action of $\GL^{\mc O}$ on $\underline \omega^\chi$.

The natural inclusion $\mc F \hookrightarrow \underline \omega_{\mathfrak Y(p^n)(\underline v)}$ is generically an isomorphism and gives an open immersion
\[
\Iwtform_{\underline w}^{\rig} \hookrightarrow \left(\mc T_K^{\an}/\U^{\mc O}_{Y_{\Iw}(p)_K^{\tor,\an}} \right)_{\mathfrak Y(p^n)(\underline v)^{\tor,\rig}}.
\]
Taking the quotient by $\U^{\mc O}(\Z_p)$, we obtain an open immersion
\[
\Iwtform_{\underline w}^{\rig,\Diamond} \hookrightarrow \left(\mc T_K^{\an}/\U^{\mc O}_{Y_{\Iw}(p)_K^{\tor,\an}} \right)_{\mathfrak Y_{\Iwt}(p^n)(\underline v)^{\tor,\rig}}.
\]
\begin{prop} \label{prop: class mod forms}
The composition of the above open immersion with the natural morphism
\[
\left(\mc T_K^{\an}/\U^{\mc O}_{Y_{\Iw}(p)_K^{\tor,\an}} \right)_{\mathfrak Y_{\Iwt}(p^n)(\underline v)^{\tor,\rig}} \to \left(\mc T_K^{\an}/\U^{\mc O}_{Y_{\Iw}(p)_K^{\tor,\an}} \right)_{\mathfrak Y_{\Iw}(p)(\underline v)^{\tor,\rig}}
\]
remains an open immersion. In particular, if $\chi$ is a dominant weight, we have a natural injective morphism
\[
\M^\chi \hookrightarrow \M^{\dagger \chi}_{\underline v,\underline w}.
\]
\end{prop}
\begin{proof}
We prove the proposition in case (A) and we leave case (C) to the reader. We have a decomposition $\Iwtform_{\underline w}^{\rig,\Diamond} = \prod_{i=1}^k \left(\Iwtform_{i,\underline w}^{+,\rig,\Diamond} \times \Iwtform_{i,\underline w}^{-,\rig,\Diamond}\right)$ and a corresponding decomposition of $\mc T_K^{\an}/\U^{\mc O}_{Y_{\Iw}(p)_K^{\tor,\an}}$. It is enough to prove the proposition for the map
\[
\Iwtform_{i,\underline w}^{\pm,\rig,\Diamond} \to \left(\mc T_{i,K}^{\pm,\an}/(\Res_{\mc O_i/\Z_p} \U_{a_i^\pm})_{Y_{\Iw}(p)_K^{\tor,\an}} \right)_{\mathfrak Y_{\Iw}(p)(\underline v)^{\tor,\rig}}.
\]
This is done explicitly as in \cite[Proposition~5.3.1]{AIP}.
\end{proof}
From now on, given $n$ and $\underline v$, we will always assume that any $\underline w \in \Q^k$ satisfies the condition in Subsection~\ref{subsec: mod sheaves}.
\section{Hecke operators} \label{sec: hecke op}
We will write $\mathcal Y(\underline v)$ for $\mathfrak Y(\underline v)^{\tor,\rig} \times_{\mathfrak Y^{\tor,\rig}} Y_K^{\an}$ and similarly for other objects. To define Hecke operators we work over $\mathcal Y_{\Iw}(p)(\underline v)$. This is enough since all the operators we are going to define send bounded functions to bounded functions, see Proposition~\ref{prop: mod form well def}.
\subsection{\texorpdfstring{Hecke operators outside $p$}{Hecke operators outside p}} \label{subsec: hecke op outside}
Recall that $N$ is a fixed positive integer not divisible by $p$ such that $\mc U^p(N) \subset \mc H$, where $\mc H$ is the level of our Shimura variety outside $p$. Let $l$ be a prime that does not divide $Np$. Let $A_1$ and $A_2$ be two abelian schemes given by the moduli problem of $Y_{\Iw}(p)_K$. An isogeny $f \colon A_1 \to f_2$ is an \emph{$l$-isogeny} if the following conditions are satisfied:
\begin{itemize}
 \item $f$ is $\mc O$-linear and its degree is a power of $l$;
 \item the pullback of the polarization of $A_2$ is a multiple of the polarization of $A_1$;
 \item the pullback of the flag of $A_2[p]$ is the flag of $A_1[p]$.
\end{itemize}
Let $f\colon A_1 \to A_2$ be an $l$-isogeny. We choose two symplectic $\mc O_{B,l}$-linear isomorphisms $\T_l(A_i) \cong \Lambda_l$, for $i=1,2$, and an isomorphism $\Z_l(1) \cong \Z_l$. In this way $f$ defines an element $\gamma \in G(\Q_l) \cap \End_{\mc O_{B,l}}(\Lambda_l) \times \Q_l^\ast$. The definition of $\gamma$ depends on the choice of the above isomorphisms, but the double class $G(\Z_l)\gamma G(\Z_l)$ depends only on $f$, and is called the \emph{type} of the $l$-isogeny $f$.

We fix a double class $G(\Z_l)\gamma G(\Z_l)$ as above. Let $C_\gamma \rightrightarrows Y_{\Iw}(p)_K$ be the moduli space that classifies $l$-isogenies $f \colon A_1 \to A_2$ of type $G(\Z_l)\gamma G(\Z_l)$, where $A_1$ and $A_2$ are abelian schemes (with additional structure) classified by $Y_{\Iw}(p)_K$. The arrow $p_j \colon C_\gamma \to Y_{\Iw}(p)_K$ send $f \colon A_1 \to A_2$ to $A_j$. Both $p_1$ and $p_2$ are finite and étale.

We fix $n$ and $\underline w$ as in the previous section. Let $C_\gamma(p^n)$ be the pullback, using $p_1$, of $C_\gamma$ to $Y(p^n)_K$. If $f \colon A_1 \to A_2$ is an isogeny parametrized by $C_\gamma(p^n)$, we can transport the trivializations of the canonical subgroups of $A_1$ to trivializations of the canonical subgroups of $A_2$, via $f$. In particular we have two finite étale morphisms $p_1,p_2 \colon C_\gamma(p^n) \rightrightarrows  Y(p^n)_K$. We write $f \colon A \to A'$ for the universal isogeny above $C_\gamma(p^n)$.

Let $\mathcal C_\gamma(p^n)(\underline v)$ be $C_\gamma(p^n)^{\an} \times_{p_1} \mathcal Y(p^n)(\underline v)$. Over $\mathcal C_\gamma(p^n)(\underline v)$, the pullback $f^\ast \colon \underline \omega_{A'} \to \underline \omega_A$ induces a morphisms $f^\ast \colon p_2^\ast \mc F \to p_1^\ast \mc F$. By Proposition~\ref{prop: F loc free}, we have that $f^\ast \colon p_2^\ast \mc F_{|\mathcal Y(p^n)(\underline v)} \stackrel{\sim}{\longrightarrow} p_1^\ast \mc F_{|\mathcal Y(p^n)(\underline v)}$ is an isomorphism. This gives a $\B^{\mc O}(\Z_p)\mathfrak B_{\underline w}^{\mc O}$-equivariant isomorphism
\[
f^\ast \colon p_2^\ast {\Iwtform_{\underline w}}_{|\mathcal Y(p^n)(\underline v)} \stackrel{\sim}{\longrightarrow} p_1^\ast {\Iwtform_{\underline w}}_{|\mathcal Y(p^n)(\underline v)}.
\]
We thus obtain a morphisms
\begin{gather} \label{eq: hecke op outside}
\Homol^0(\mathcal Y(p^n)(\underline v), \mc O_{\Iwtform_{\underline w}}) \stackrel{p_2^{\ast}}{\longrightarrow} \Homol^0(\mathcal C_\gamma(p^n)(\underline v), p_2^\ast \mc O_{\Iwtform_{\underline w}}) \stackrel{(f^\ast)^{-1,\rig}}{\longrightarrow} \\
\stackrel{(f^\ast)^{-1}}{\longrightarrow} \Homol^0(\mathcal C_\gamma(p^n)(\underline v), p_1^\ast\mc O_{\Iwtform_{\underline w}}) \stackrel{\tr p_1^{\rig}}{\longrightarrow} \Homol^0(\mathcal Y(p^n)(\underline v), \mc O_{\Iwtform_{\underline w}}). \nonumber
\end{gather}
\begin{defi} \label{defi: hecke op outside p}
Let $\chi \in \mc W(\underline w)(K)$ be a character. We define the operator $\T_\gamma \colon \M^{\dagger \chi}_{\underline v,\underline w} \to \M^{\dagger \chi}_{\underline v,\underline w}$ from the composition defined in \eqref{eq: hecke op outside} considering bounded and homogeneous sections, for the action of $\B^{\mc O}(\Z_p)\B_{\underline w}^{\mc O}$, of degree $\chi'$. We obtain an operator
\[
\T_\gamma \colon \M^{\dagger \chi} \to \M^{\dagger \chi}.
\]
Since $p$ is unramified in $B$, the operator $\T_\gamma$, for various $\gamma$ and $l$ commute. We let $\mathds T^{Np}$ be the restricted tensor product of the algebras
\[
\Z[G(\Z_l)/ (G(\Q_l) \cap \End_{\mc O_{B,l}}(\Lambda_l) \times \Q_l^\ast )\backslash G(\Z_l)]
\]
for $l$ a prime with $(l,Np)=1$. We have defined an action of $\mathds T^{Np}$ on $\M^{\dagger \chi}_{\underline v, \underline w}$ and on $\M^{\dagger \chi}$. 
\end{defi}
\subsection{\texorpdfstring{Hecke operators at $p$}{Hecke operators at p}} \label{subsec: hecke op at}
In this subsection we fix an index $i=1,\ldots,k$. The operators we are going to define will acts as the identity outside the $i$-th component. We assume that $\underline v$ satisfies $v_i < \frac{p-2}{2p-2}$ and that $\underline w$ is as above.
\subsubsection{\texorpdfstring{The operator $\U^\pm_{i,a_i^\pm}$}{The operator \U^\pm_{i,a_i^\pm}}} \label{subsub: Ug}
We start by defining an operator $\U^\pm_{i,a_i^\pm} = \U^+_{i,a_i^+}=\U^-_{i,a_i^-}$ in case (A) and an operator $\U_{i,a_i}$ in case (C) (this notation will be clear later on). In case (A), let $p_1,p_2 \colon \mathcal C(\underline v)^\pm_{i,a_i^\pm} \rightrightarrows \mathcal Y_{\Iw}(p)(\underline v)$ be the moduli space that classifies couples $(A,L_i^\pm)$ where $A$ is an abelian scheme classified by $\mathcal Y_{\Iw}(p)(\underline v)$ and $L_i^\pm \subset G_i^\pm[(\varpi_i^\pm)^{e_i}]$ is a finite and flat subgroup, stable under $\mc O_i$, and such that $G_i^\pm[(\varpi_i^\pm)^{e_i}] = H_{i,1}^\pm \oplus L_i^\pm$ (note that $L_i^+ \mapsto L_i^{+,\perp}$ gives a canonical isomorphism between $C^+_{i,a_i^+}$ and $C^-_{i,a_i^-}$). In case (C) we make a similar definition, adding the condition that $L_i$ is totally isotropic for the polarization of $G_i$. By Lemma~\ref{lemma: can sub free}, any $L_i^\pm$ as above is étale locally isomorphic to $(\mc O_i / (\varpi_i^\pm)^{e_i})^{a_i^\pm}$. The arrow $p_1 \colon C^\pm_{i,a_i^\pm} \to \mathcal Y_{\Iw}(p)(\underline v)$ forgets $L_i^\pm$ and $p_2$ is defined taking the quotient (via Morita's equivalence) by $L_i^+\oplus L_i^-$ or by $L_i$. The map $p_1$ is finite and étale. By \cite[Proposition~16]{fargues_can}, we have that $p_2$ gives a morphism, denoted again $p_2 \colon \mathcal C(\underline v)^\pm_{i,a_i^\pm} \to \mathcal Y_{\Iw}(p)(\underline v')$, where $\underline v' = (v')_i$ is defined by $v'_j=v_j$ if $j \neq i$ and $v_i'=v_i/p$.

We write $\mathcal C(p^n)(\underline v)^\pm_{i,a_i^\pm}$ for the pullback, using $p_1$, of $\mathcal C(\underline v)^\pm_{i,a_i^\pm}$ to $\mathcal Y(p^n)(\underline v)$. We have two natural morphisms $p_1 \colon \mathcal C(p^n)(\underline v)^\pm_{i,a_i^\pm} \to \mathcal Y(p^n)(\underline v)$ and $p_2 \colon \mathcal C(p^n)(\underline v)^\pm_{i,a_i^\pm} \to \mathcal Y(p^n)(\underline v')$. Moreover, over $\mathcal C(p^n)(\underline v)^\pm_{i,a_i^\pm}$, we have an isomorphism $f^\ast \colon p_2^\ast \mc F^{\rig} \stackrel{\sim}{\longrightarrow} p_1^\ast \mc F^{\rig}$ and a $\B^{\mc O}(\Z_p)\mathfrak B_{\underline w}^{\mc O}$-equivariant isomorphism (that is the identity outside the $i$-th component)
\[
f^\ast \colon p_2^\ast {\Iwtform_{\underline w}^{\rig}}_{|\mathcal Y(p^n)(\underline v')} \to p_1^\ast {\Iwtform_{\underline w}^{\rig}}_{|\mathcal Y(p^n)(\underline v)}.
\]
We thus obtain a morphism
\begin{gather*}
\Homol^0(\mathcal Y(p^n)(\underline v'), \mc O_{\Iwtform_{\underline w}}) \stackrel{p_2^{\ast}}{\longrightarrow} \Homol^0(\mathcal C(p^n)(\underline v)^\pm_{i,a_i^\pm}, p_2^\ast \mc O_{\Iwtform_{\underline w}}) \stackrel{(f^\ast)^{-1}}{\longrightarrow} \\
\stackrel{(f^\ast)^{-1}}{\longrightarrow} \Homol^0(\mathcal C(p^n)(\underline v)^\pm_{i,a_i^\pm}, p_1^\ast\mc O_{\Iwtform_{\underline w}}) \stackrel{\tr p_1^{\rig}}{\longrightarrow} \Homol^0(\mathcal Y(p^n)(\underline v), \mc O_{\Iwtform_{\underline w}^{\rig}}).
\end{gather*}
Taking the composition we get, for any $\chi \in \mc W(\underline w)$, an operator
\[
\widetilde{\U}_{i,a_i^\pm}^\pm \colon \M^{\dagger \chi}_{\underline v',\underline w} \to \M^{\dagger \chi}_{\underline v,\underline w}.
\]
We define
\[
\U_{i,a_i^\pm}^\pm \colonequals \left(\frac{1}{p} \right)^{d_ia_i^+a_i^-}\widetilde{\U}_{i,a_i^\pm}^\pm
\]
in case (A) and
\[
\U_{i,a_i} \colonequals \left(\frac{1}{p} \right)^{\frac{d_ia_i(a_i+1)}{2}}\widetilde{\U}_{i,a_i}
\]
in case (C) (these are the normalization factors of \cite{stefan}). We will use the same symbols to denote the composition with $\M^{\dagger \chi}_{\underline v,\underline w} \hookrightarrow \M^{\dagger \chi}_{\underline v',\underline w}$. The operators
\[
\U_{i,a_i^\pm}^\pm \colon \M^{\dagger \chi} \to \M^{\dagger \chi}
\]
are completely continuous.
\subsubsection{\texorpdfstring{The operators $\U^\pm_{i,j}$}{The operators \U^\pm_{i,j}}} \label{subsub: Uj}
Additionally to the various assumptions we made above, we assume in this subsection that $\underline v$ satisfies $v_i < \frac{p-2}{2p^2-p}$. We explain the construction in case (A) and we leave case (C) to the reader. Let $\overline w_i^\pm = (w_i^{\pm,r,s})_{1\leq s \leq r \leq a_i^\pm}$ be a $\frac{a_i^\pm(a_i^\pm + 1)}{2}$-tuple of rational numbers such that $w_i^{\pm,r,s} \in ]\frac{v_i}{p-1}, n-v_i\frac{p^n}{p-1}]$. We moreover assume that $w_i^{\pm,r+1,s} \geq w_i^{\pm,r,s}$ and $w_i^{\pm,r,s+1} \leq w_i^{\pm,r,s}$. We define $\Iwtform_{i,\overline w_i^\pm}^\pm \to \mathfrak Y(p^n)(\underline v)$ as follows. Let $R$ be in $\Nadm$ and suppose that $\mc F_i^\pm(R)$ is free. The $R$-points of $\Iwtform_{i,\overline w_i^\pm}^\pm$ correspond to the following data:
\begin{itemize}
 \item an $R$-point of $\mathfrak Y(p^n)(\underline v)$,
 \item a filtration
\[
\Fil_\bullet \mc F_i^\pm(R)=(0=\Fil_0 \mc F_i^\pm(R) \subset \cdots \subset \Fil_{a_i^\pm}\mc F_i^\pm(R) = \mc F_i^\pm(R)); 
\]
 \item trivializations
\[
\omega_i^r \colon \Gr_r \Fil_\bullet \mc F_i^\pm(R) \stackrel{\sim}{\longrightarrow} (\mc O_i \otimes_{\Z} R)^r
\]
\end{itemize}
such that the following conditions hold:
\begin{itemize}
 \item $\Fil_r \mc F_i^\pm(R)$ is a free $\mc O_i \otimes_{\Z} R$-module for each $0 \leq r \leq a_i^\pm$;
 \item let $e_1,\ldots, e_{a_i^\pm}$ be the $R$-points of $(H_{i,n}^\pm)^{\Dual}$ defined by the given isomorphism $(H_{i,n}^\pm)_K^{\Dual} \cong (\mc O_i/p^n)^{a_i^\pm}$ and set $w \colonequals n-v_i\frac{p^n}{p-1}$. We require that the following equality holds, in $\Fil_r \mc F_i^\pm (R)/(\Fil_{r-1} \mc F_i^\pm (R) + p^w)$,
\[
\omega_i^r = \sum_{r \geq s} a_{r,s} \HT_{i,w}^\pm(e_k), 
\]
where $a_{r,s} \in R$ are such that $a_{r,s} \in p^{w_i^{\pm,r,s}}R$ if $r > s$ and $(a_{r,r} -1) \in p^{w_i^{\pm,r,r}} R$.
\end{itemize}
If $w_i^{\pm,r,s}=w_i^\pm$ are all equal, the definition of $\Iwtform_{i,\overline w_i^\pm}^\pm$ is exactly the same as the definition of $\Iwtform_{i,w_i^\pm}^\pm$. All the constructions we did for $\Iwtform_{i,w_i^\pm}^\pm$ generalize to $\Iwtform_{i,\overline w_i^\pm}^\pm$ and we use the corresponding notation. For example $\Iwtform_{i,\overline w_i^\pm}^\pm$ extends to $\mathfrak Y(p^n)(\underline v)^{\tor}$ and we have $\Iwtform_{\overline{\underline w}} \to \mathfrak Y(p^n)(\underline v)^{\tor}$ or $\Iwtform_{\overline {\underline w}}^{\rig,\Diamond} \to \mathfrak Y_{\Iw}(p)(\underline v)^{\tor,\rig}$ (here we fix a tuple $\overline w^i$ as above for each $i$). Looking at the proof of Proposition~\ref{prop: class mod forms}, one sees that the condition $w_i^{\pm,r,s} \in ]\frac{v_i}{p-1}, n-v_i\frac{p^n}{p-1}]$ implies that the natural map $\Iwtform_{\overline{ \underline w}}^{\rig} \to \left ( \mc T_K^{\an}/\U^{\mc O}_{Y_{\Iw}(p)_K^{\tor,\an}} \right)_{\mathfrak Y_{\Iw}(p)(\underline v)^{\tor,\rig}}$ is an open immersion. If $\chi \in \mc W(\underline w)$, we can define as above the sheaf $\underline{\mathfrak w}^{\dagger\chi}_{\underline v, \underline{\overline w}}$ and its rigid fiber $\underline \omega^{\dagger\chi}_{\underline v, \underline{\overline w}}$. 	

We now fix $1 \leq j < a_i^\pm$ and we define the operator $\U^\pm_{i,j}$. Let $p_1,p_2 \colon \mathcal C(\underline v)^\pm_{i,j} \rightrightarrows \mathcal Y_{\Iw}(p)(\underline v)$ be the moduli space that classifies couples $(A,L_i^\pm)$ where $A$ is an abelian scheme classified by $\mathcal Y_{\Iw}(p)(\underline v)$ and $L_i^\pm \subset G_i^\pm[p^2]$ is a finite and flat subgroup, stable under $\mc O_i$, and such that $G_i^\pm[p] = \Fil_j H_{i,1}^\pm \oplus L_i^\pm[p]$. The arrow $p_1 \colon \mathcal C(\underline v)^\pm_{i,j} \to \mathcal Y_{\Iw}(p)(\underline v)$ forgets $L_i^\pm$ and $p_2$ is defined taking the quotient (via Morita's equivalence) by $L_i^+\oplus L_i^-$ (by \cite[Proposition~6.2.2.1]{AIP}, the image of $(A,L_i^\pm)$ by $p_2$ lies in $\mathcal Y_{\Iw}(p)(\underline v)$). Let $f \colon A \to A'$ be the universal isogeny over $\mathcal C(\underline v)^\pm_{i,j}$. It gives a $\B^{\mc O}(\Z_p)\mathfrak B^{\mc O}_{\overline{\underline w}}$ isomorphism
\[
f^\ast \colon p_2^\ast \left(\mc T_K^{\an}/\U^{\mc O}_{Y_{\Iw}(p)_K^{\tor,\an}}\right)_{|\mathcal Y_{\Iw}(p)(\underline v)} \stackrel{\sim}{\longrightarrow} p_1^\ast \left ( \mc T_K^{\an}/\U^{\mc O}_{Y_{\Iw}(p)_K^{\tor,\an}}\right)_{|\mathcal Y_{\Iw}(p)(\underline v)}.
\]
Let $\overline w_i^\pm=(w_i^{\pm,r,s})_{r,s}$ be a tuple as above with the additional condition that $w_i^{\pm,r,k} < n-2 - v_i\frac{p^n}{p-1}$. We definite a tuple ${\overline w_i^\pm}' = ({w_i'}^{\pm,r,s})_{r,s}$ by
\[
{w_i'}^{\pm,r,s} = 
\begin{cases}
w_i^{\pm,r,s} + 1 \mbox{ if }  r \geq j+1 \mbox{ and } s \leq j\\
w_i^{\pm,r,s} \mbox{ otherwise}
\end{cases}
\]
Starting with $\overline{\underline w}$, we define $\overline{\underline w}'$ modifying only the $i$-th component $\overline w_i^\pm$. It follows that we have the spaces $\Iwtform_{\overline{\underline w}}^{\rig}$ and $\Iwtform_{\overline{\underline w}'}^{\rig}$ and both are open subset of $\mc T_K^{\an}/\U^{\mc O}_{Y_{\Iw}(p)_K^{\tor,\an}}$. The proof of \cite[Proposition~6.2.2.2]{AIP} works also in our case, so we have
\[
(f^{\ast})^{-1} p_1^\ast \left( \Iwtform_{\overline{\underline w}}^{\rig} \right)_{|\mathcal Y_{\Iw}(p)(\underline v)} \subset p_2^\ast \left( \Iwtform_{\overline{\underline w}'}^{\rig}\right)_{|\mathcal Y_{\Iw}(p)(\underline v)}.
\]
For $\chi \in \mc W(\underline w)(K)$, we can now define an operator
\[
\U^\pm_{i,j} \colon \M^{\dagger \chi}_{\underline v, \overline{\underline w}'} \to \M^{\dagger \chi}_{\underline v, \overline{\underline w}}
\]
using the composition
\begin{gather*}
\Homol^0(\mathcal Y_{\Iw}(p)(\underline v), \underline \omega^{\dagger \chi}_{\underline v, \overline{\underline w}'}) \stackrel{p_2^{\ast}}{\longrightarrow} \Homol^0(\mathcal C(\underline v)^\pm_{i,j}, p_2^\ast \underline \omega^{\dagger \chi}_{\underline v, \overline{\underline w}'}) \stackrel{(f^\ast)^{-1}}{\longrightarrow} \\
\stackrel{(f^\ast)^{-1}}{\longrightarrow} \Homol^0(\mathcal C(\underline v)^\pm_{i,j}, p_1^\ast \underline \omega^{\dagger \chi}_{\underline v, \overline{\underline w}}) \stackrel{\tr p_1^{\rig}}{\longrightarrow} \Homol^0(\mathcal Y_{\Iw}(p)(\underline v), \underline \omega^{\dagger \chi}_{\underline v, \overline{\underline w}}).
\end{gather*}
We also have operators
\[
\U^\pm_{i,j} \colon \M^{\dagger \chi}_{\underline v, \overline{\underline w}} \to \M^{\dagger \chi}_{\underline v, \overline{\underline w}} \mbox{ and } \U^\pm_{i,j} \colon \M^{\dagger \chi} \to \M^{\dagger \chi}.
\]
\subsubsection{\texorpdfstring{The $\U$-operator}{The operators \U-operator}} \label{subsub: U} We work in case (A), case (C) is similar. We fix $n$. From now on, we will always assume that the following conditions are satisfied. Let $\underline v$ be such that the above inequalities hold and let $\underline w$ be such that $w_i^\pm \in ]\frac{v_i}{p-1},n-1 - a_i^\pm]$.

Let us fix $i= 1,\dots, k$. We set
\begin{gather*}
v_i' \colonequals v_i/p \mbox{ and } v_j' \colonequals v_j \mbox{ if } j \neq i, \\
w_j^{\pm,r,s} \colonequals w_j^\pm \mbox{ for all } r,s \mbox{ and for all } j,\\
{w_i'}^{\pm,r,s} \colonequals r - s + w_i^\pm \mbox{ and } {w_j'}^{\pm,r,s} \colonequals w_i^\pm \mbox{ if } j \neq i.
\end{gather*}
and for the rest we use the above notations. The product 
\[
\U_i \colonequals \U_{i,a_i^\pm}^\pm \times \prod_{j=1}^{a_i^\pm - 1} (\U_{i,j}^+ \times \U_{i,j}^-)
\]
gives an operator $\U_i \colon \M^{\dagger\chi}_{\underline v', \overline {\underline w}'} \to \M^{\dagger\chi}_{\underline v, \overline {\underline w}} = \M^{\dagger\chi}_{\underline v,\underline w}$. We denote with the same symbol the composition of $\U_i$ with the natural restriction $\M^{\dagger\chi}_{\underline v, \underline w} = \M^{\dagger\chi}_{\underline v, \overline {\underline w}} \to \M^{\dagger\chi}_{\underline v', \overline {\underline w}'}$, obtaining
\[
\U_i \colon \M^{\dagger\chi}_{\underline v, \underline w} \to  \M^{\dagger\chi}_{\underline v,\underline w}.
\]
Taking the product of the $\U_i$ we obtain the compact operators
\[
\U \colon \M^{\dagger\chi}_{\underline v, \underline w} \to  \M^{\dagger\chi}_{\underline v,\underline w}. \mbox{ and } \U \colon \M^{\dagger\chi} \to  \M^{\dagger\chi}.
\]
\begin{rmk} \label{rmk: class}
Usually one defines the $\U$-operator using only the operators $\U^\pm_{i,a_i^\pm}$, that improve the degree of overconvergence. The reasons for including also the operators $\U^\pm_{i,j}$ is that they improve analyticity, and this will be needed to prove classicity in Section~\ref{sec: class}.
\end{rmk}
We let $\m U_p$ be the free $\Z$-algebra generated by the Hecke operators at $p$ and let $\m T \colonequals \m T^{Np} \otimes_{\Z} \m U_p$. We simply call $\m T$ the Hecke algebra. It acts on all the spaces we have defined.  
\subsection{Families} \label{subsec: families}
Let $\mc U \subset \mc W$ be an affinoid associated to the algebra $A \colonequals \mc O_{\mc U}(\mc U)$. There is $\underline w$ such that $\mc U \subset \mc W(\underline w)$, and we fix one such. We write $\chi^{\un}_{\mc U} \colon \T^{\mc O}(\Z_p) \to A^\ast$ for the universal character over $\mc U$. Let $\underline v$ and $n$ be as usual.
\begin{prop} \label{prop: families}
There is a Banach sheaf $\underline \omega_{\underline v, \underline w}^{\dagger \chi^{\un}_{\mc U}}$ on $\mathfrak Y_{\Iw}(p)(\underline v)^{\tor,\rig} \times \mc U$ such that, for any $\chi \in \mc U(K)$, the fiber of $\underline \omega_{\underline v, \underline w}^{\dagger \chi^{\un}_{\mc U}}$ at $\mathfrak Y_{\Iw}(p)(\underline v)^{\tor,\rig} \times \set{\chi}$ is canonically isomorphic to $\underline \omega_{\underline v, \underline w}^{\dagger \chi}$. On the global sections of $\underline \omega_{\underline v, \underline w}^{\dagger \chi^{\un}_{\mc U}}$, there is an action of the Hecke operators defined above.
\end{prop}
\begin{proof}
We have a morphism $\pi_1 \times \operatorname{id} \colon \Iwtform_{\underline w}^{\Diamond,\rig} \times \mc U \to \mathfrak Y_{\Iw}(p)(\underline v)^{\tor,\rig} \times \mc U$. On
\[
(\pi_1 \times \operatorname{id})_\ast \mc O_{\Iwtform_{\underline w}^{\Diamond,\rig} \times \mc U}
\]
there is an action of $\B^{\mc O}(\Z_p)\B_{\underline w}$, and we define $\underline \omega_{\underline v, \underline w}^{\dagger \chi^{\un}}$ taking sections homogeneous of degree $(\chi^{\un}_{\mc U})'$. The definitions given above of the Hecke operators work in families without problems.
\end{proof}
\begin{defi} \label{defi: fam mod form}
We define
\[
\M^{\dagger\mc U}_{\underline v, \underline w} \colonequals \Homol^0(\mathfrak Y_{\Iw}(p)(\underline v)^{\tor,\rig} \times \mc U,\underline \omega_{\underline v, \underline w}^{\dagger \chi^{\un}_{\mc U}}).
\]
It is the space of families of $\underline v$-overconvergent $\underline w$-locally analytic modular forms parametrized by $\mc U$. We set
\[
\M^{\dagger\mc U} \colonequals \lim_{\substack{\underline v \to  0\\ \underline w \to \infty}} \M^{\dagger\mc U}_{\underline v, \underline w}.
\]
It is the space of overconvergent locally analytic modular forms parametrized by $\mc U$.
\end{defi}
We have an action of the Hecke operators on both $\M^{\dagger\mc U}_{\underline v, \underline w}$ and $\M^{\dagger\mc U}$. The $\U$-operator on $\M^{\dagger\mc U}$ is completely continuous. We have that $\M^{\dagger\mc U}_{\underline v, \underline w}$ is a Banach $A$-module.
\begin{rmk} \label{rmk: spec}
Let $\chi \in \mc U(K)$. Then we have a natural specialization map
\[
\M^{\dagger\mc U}_{\underline v, \underline w} \to \M^{\dagger\chi}_{\underline v, \underline w}.
\]
We do not know whether this map is surjective or not (but see Corollary~\ref{coro: proj} below for the cuspidal case).
\end{rmk}
\begin{rmk} \label{rmk: int fam}
We can define a formal model $\mathfrak W$ and $\mathfrak W(\underline w)$ of $\mc W$ and $\mc W(\underline w)$. If $\mathfrak{U} \subset \mathfrak{W}(\underline w)$ is affine, we have an analogue of Proposition~\ref{prop: families}, obtaining the formal Banach sheaf $\underline{\mathfrak w}_{\underline v, \underline w}^{\dagger \chi^{\un}_{\mathfrak U}}$ on $\mathfrak Y_{\Iw}(p)(\underline v)^{\tor} \times \mathfrak U$. 
\end{rmk}

\section{Cuspidal forms and eigenvarieties} \label{sec: heck var}
\subsection{The minimal compactifications}
Recall that in Section~\ref{sec: mod forms} we have fixed a compatible choice of admissible smooth rational polyhedral cone decomposition data for $Y$. This choice gives the toroidal compactification $Y^{\tor}$ of $Y$. We also have the minimal compactification $Y^{\min}$ and a proper morphism
\[
\xi \colon Y^{\tor} \to Y^{\min},
\]
see \cite[Theorem~7.2.4.1]{lan} for the main properties of $Y^{\min}$. The sheaves $\underline \omega_i$ extend to the minimal compactification, and the Hasse invariants give sections $\Ha_i$ of $\det(\underline \omega_i)^{\otimes p-1}$ on the special fiber of $Y^{\min}$. We then obtain a function $\Hdg \colon \mathfrak Y^{\min,\rig} \to [0,1]^k$ and so the rigid space $\mathfrak Y(\underline v)^{\min,\rig}$ and its formal model $\mathfrak Y(\underline v)^{\min}$ are defined.

Let $\Sp(R_{\alg})$ be part of the data giving a good algebraic model as in the beginning of Section~\ref{sec: mod forms}. In particular we have the semiabelian schemes $A \to \Sp(R)$ and $\tilde A \to \Sp(R)$. Since the formal completions of $A$ and $\tilde A$ along the closed stratum of $S = \Sp(R_{\alg})$ are isomorphic, we have an isomorphism of locally free sheaves over $S$
\[
e_A^\ast \Omega^1_{A/S} \cong e_{\tilde A}^\ast \Omega^1_{\tilde A/S},
\]
where $e_A$ and $e_{\tilde A}$ are the corresponding zero sections. Hasse invariants are compatible with respect to the induced isomorphisms, so $\xi$ gives a morphism
\[
\xi(\underline v) \colon \mathfrak Y(\underline v)^{\tor} \to \mathfrak Y(\underline v)^{\min}.
\]
We will write $D$ for both the boundary of $Y^{\tor}$ and the boundary of $\mathfrak Y(\underline v)^{\tor}$. We then have the following
\begin{teo}[{\cite[Theorem~8.2.1.2]{lan_ram}}] \label{teo: ann coho}
We have
\[
\R^q \xi_\ast \mc O_{Y^{\tor}}(-D) = 0
\]
if $q \geq 1$.
\end{teo}
\begin{notation}
Let $\mathfrak m$ be the maximal ideal of $\mc O_K$. If $\star$ is an object defined over $\mc O_K$, we let $\star_n$ be the reduction of $\star$ modulo $\mathfrak m^n$.
\end{notation}
\begin{coro} \label{coro: ann coho v}
We have
\[
\R^q \xi(\underline v)_\ast \mc O_{\mathfrak Y(\underline v)^{\tor}}(-D) = 0
\]
if $q \geq 1$.
\end{coro}
\begin{proof}
Arguing as in \cite[Proposition~8.2.1.2]{AIP} we have that the description of the formal fibers of $\xi(\underline v)$ is the same as the description given in \cite[Section~8.2]{lan_ram}, hence to prove the corollary one can repeat the proof of \cite[Theorem~8.2.1.2]{lan_ram}.
\end{proof}
Let $\eta(\underline v)$ be the composition 
\[
\mathfrak Y(p^n)(\underline v)^{\tor} \to \mathfrak Y(\underline v)^{\tor} \to \mathfrak Y(\underline v)^{\min}
\]
and let $\rho(\underline v)$ be the first morphism. We will still write $D$ for its inverse image under $\eta(\underline v)$.
\begin{prop} \label{prop: vani fin}
We have
\[
\R^q \eta(\underline v)_\ast \mc O_{\mathfrak Y(p^n)(\underline v)^{\tor}}(-D) = 0
\]
if $q \geq 1$.
\end{prop}
\begin{proof}
This can be proved by the same arguments used in \cite[Section~8.2]{lan}. To do this, we need to describe the local charts of $\mathfrak{Y}(p^n)(\underline v)^{\tor}$ similarly to the one of $\mathfrak{Y}(\underline v)^{\tor}$. Over the rigid fiber $\mathfrak{Y}(p^n)(\underline v)^{\tor,\rig}$ we know that such a description exists. We can now argue as in \cite[Proposition~8.2.1.3]{AIP} to extend this description to $\mathfrak{Y}(p^n)(\underline v)^{\tor}$.
\end{proof}
\subsection{A dévissage} \label{subsec: dev}
Let $\mc W(\underline w)^\circ$ be the rigid open unit disk of dimension $\dim(\T_{\underline w}^{\mc O})$. As in \cite[Section~2.2]{AIP}, we have an analytic universal character
\[
\chi^{\un,\circ} \colon \mathfrak{T}_{\underline w}^{\mc O}(\Z_p) \to \mc O_{\mc W(\underline w)^\circ}(\mc W(\underline w)^\circ)^\ast.
\]
We set $\mathfrak{W}(\underline w)^\circ \colonequals \Spf(\mc O_K \llbracket X_1,\ldots,X_{\dim(\T_{\underline w}^{\mc O})} \rrbracket )$, a formal model of $\mc W(\underline w)^\circ$. Using the construct of \cite[Section~2.2]{AIP}, we see that the character $\chi^{\un,\circ}$ is induced by a formal universal character, denoted in the same way,
\[
\chi^{\un,\circ} \colon \mathfrak{T}_{\underline w}^{\mc O}(\Z_p) \to \mc O_{\mathfrak{W}(\underline w)^\circ}(\mathfrak{W}(\underline w)^\circ)^\ast.
\]
Recall that we have the morphism $\zeta \colon \Iwtform_{\underline w} \to \mathfrak Y(p^n)(\underline v)^{\tor}$. The torus $\mathfrak{T}_{\underline w}^{\mc O}$ acts on this space, so, if $\chi^\circ \in \mc W(\underline w)^\circ(K)$, we can define the sheaves on $\mathfrak Y(p^n)(\underline v)^{\tor}$
\[
\underline{\mathfrak w}^{\dagger \chi^\circ}_{\underline v,\underline w} \colonequals \zeta_\ast \mc O_{\Iwtform_{\underline w}}[{\chi^\circ}'].
\]
As in Remark~\ref{rmk: int fam}, if $\mathfrak U^\circ \subset \mathfrak W(\underline w)^\circ$ is affine, with associated character $\chi_{\mathfrak U^\circ}^{\un}$, we have the family $\underline{\mathfrak w}^{\dagger \chi_{\mathfrak U^\circ}^{\un}}_{\underline v,\underline w}$. It is a sheaf on $\mathfrak Y(p^n)(\underline v)^{\tor} \times \mathfrak{U}$.
\begin{rmk} \label{rmk: dev is dev}
If $\chi \in \mc W(\underline w)(K)$, then we have its image $\chi^\circ \in \mc W(\underline w)^\circ(K)$ and we can recover $\underline{\mathfrak w}^{\dagger \chi}_{\underline v,\underline w}$ from $\underline{\mathfrak w}^{\dagger \chi^\circ}_{\underline v,\underline w}$ projecting to $\mathfrak Y_{\Iw}(p)(\underline v)^{\tor}$ and taking homogeneous sections of degree $\chi'$ for the action of $\B^{\mc O}(\Z_p)\mathfrak{B}_{\underline w}^{\mc O}$. Equivalently, one can project $\underline{\mathfrak w}^{\dagger \chi^\circ}_{\underline v,\underline w}$ to $\mathfrak Y_{\Iw}(p)(\underline v)^{\tor}$, twist the action of $\B^{\mc O}(\Z_p)\mathfrak{B}_{\underline w}^{\mc O}$ by $-\chi'$, and finally take invariant sections for the action of $\B^{\mc O}(\Z/p^n\Z)$. A similar remark holds for families.
\end{rmk}
We recall that $\star_n$ means the reduction modulo $\mathfrak{m}^n$ of $\star$, where $\mathfrak{m}$ is the maximal ideal of $\mc O_K$.
\begin{prop} \label{prop: small ban}
Let $\chi^\circ \in \mc W(\underline w)^\circ(K)$ be a character. Then $\underline{\mathfrak w}^{\dagger \chi^\circ}_{\underline v,\underline w, 1}$ is an inductive limit of coherent sheaves that are an extension of the trivial sheaf. In particular, $\underline{\mathfrak w}^{\dagger \chi^\circ}_{\underline v,\underline w}$ is a small formal Banach sheaf (see \cite[Definition~A.1.2.1]{AIP} for the definition of small formal Banach sheaf). If $\mathfrak U^\circ \subset \mathfrak W(\underline w)^\circ$ is affine, a similar result hold for $\underline{\mathfrak w}^{\dagger \chi_{\mathfrak U^\circ}^{\un}}_{\underline v,\underline w}$.
\end{prop}
\begin{proof}
We prove the proposition for $\underline{\mathfrak w}^{\dagger \chi^\circ}_{\underline v,\underline w}$, the proof for $\underline{\mathfrak w}^{\dagger \chi_{\mathfrak U^\circ}^{\un}}_{\underline v,\underline w}$ is similar. Using the decompositions $\Iwtform_{\underline w} = \prod_{i=1}^k (\Iwtform_{i,w_i^+}^+ \times \Iwtform_{i,w_i^-}^-)$ in case (A) and $\Iwtform_{\underline w} = \prod_{i=1}^k \Iwtform_{i,w_i}$ in case (C), we can prove the proposition for the sheaf
\[
\mc G \colonequals \zeta_{i,\ast}^\pm \mc O_{\Iwtform_{i,w_i^\pm}^\pm}[\chi'],
\]
where $\zeta_i^\pm \colon \Iwtform_{i,w_i^\pm}^\pm \to \mathfrak{Y}(p^n)(\underline v)^{\tor}$ is the natural morphism. One can give a completely explicit description of the sections of $\mc G$, as in \cite[Subsections~8.1.5 and 8.1.6]{AIP}. The proof is then identical to the one of \cite[Corollary 8.1.6.2]{AIP}.
\end{proof}
Recall the morphism $\eta(\underline v) \colon \mathfrak Y(p^n)(\underline v)^{\tor} \to \mathfrak Y(\underline v)^{\min}$.
\begin{prop} \label{prop: proj small ban}
For any $\chi^\circ \in \mc W(\underline w)^\circ(K)$ and any affine $\mathfrak U^\circ \subset \mathfrak W(\underline w)^\circ$ we have that
\[
\eta(\underline v)_\ast (\underline{\mathfrak w}^{\dagger \chi^\circ}_{\underline v,\underline w}(-D)) \mbox{ and } (\eta(\underline v \times \operatorname{id}))_\ast (\underline{\mathfrak w}^{\dagger \chi_{\mathfrak U^\circ}^{\un}}_{\underline v,\underline w} (-D))
\]
are small formal Banach sheaves
\end{prop}
\begin{proof}
We prove the statement for $\eta(\underline v)_\ast (\underline{\mathfrak w}^{\dagger \chi^\circ}_{\underline v,\underline w}(-D))$, the case of families is similar. We use the notation of the proof of Proposition~\ref{prop: small ban}. It is enough to prove the proposition for $\eta(\underline v)_\ast (\mc G(-D))$. Let $s \geq 1$ be an integer and consider the commutative diagram
\[
\xymatrix{
\mathfrak{Y}(p^n)(\underline v)^{\tor}_{s-1} \ar[r]^i \ar[d]^{\eta(\underline v)_{s-1}} & \mathfrak{Y}(p^n)(\underline v)^{\tor}_s \ar[d]^{\eta(\underline v)_s} \\
\mathfrak{Y}(\underline v)^{\min}_{s-1} \ar[r]^j  & \mathfrak{Y}(\underline v)^{\min}_s
}
\]
where $i$ and $j$ are the natural closed immersions. Using Proposition~\ref{prop: vani fin} and \cite[Proposition~A.1.3.1]{AIP}, we see that it is enough to prove that
\[
j^\ast (\eta(\underline v)_{s,\ast} \mc G(-D)_s) = \eta(\underline v)_{s-1.\ast} \mc G(-D)_{s-1}.
\]
Similarly to $\mc G_1$, we can write $\mc G_s \cong \varinjlim_j \mc G_{s,j}$, where each $\mc G_{s,j}$ is a coherent sheaf. Since taking direct images commutes with direct limits, we can prove the statement for $\mc G_{s,j}$. From now on, all the sheaves will be considered as sheaves over $\mathfrak{Y}(p^n)(\underline v)^{\tor}_s$. It is enough to prove that 
\[
\R^1 \eta(\underline v)_{s, \ast} \ker(\mc G_{s,j} \to \mc G_{s-1,j}) = 0.
\]
We have (see \cite[Section~8.1.6]{AIP} and the proof of Proposition~\ref{prop: small ban})
\[
\ker(\mc G_{s,j} \to \mc G_{s-1,j}) \cong \mc G_{1,j},
\]
in particular we need to show that $\R^1 \eta(\underline v)_{s, \ast} \mc G_{1,j} = 0$. By Proposition~\ref{prop: small ban}, we know that $\mc G_{1,j}$ is an extension of the trivial sheaf, so we can conclude by Proposition~\ref{prop: vani fin}.
\end{proof}
\subsection{Projectivity of the space of cuspidal forms} \label{subsec: proj}
We fix $\mathfrak U^\circ \subset \mathfrak W(\underline w)^\circ$ an affine with rigid fiber $\mc U^\circ = \Spm(A^\circ)$ and $\chi^\circ \in \mc U^\circ(K)$. Let us consider the modules $\cusp^{\mc U^\circ}_{\underline v,\underline w}$ and $\cusp^{\chi^\circ}_{\underline v,\underline w}$ defined by
\[
\cusp^{\mc U^\circ}_{\underline v,\underline w} \colonequals \Homol^0(\mathfrak Y(p^n)(\underline v)^{\tor} \times \mathfrak U^\circ, \underline{\mathfrak w}^{\dagger \chi_{\mathfrak U^\circ}^{\un}}_{\underline v,\underline w} (-D))[p^{-1}]
\]
and
\[
\cusp^{\chi^\circ}_{\underline v,\underline w} \colonequals \Homol^0(\mathfrak Y(p^n)(\underline v)^{\tor}, \underline{\mathfrak w}^{\dagger \chi^\circ}_{\underline v,\underline w}(-D))[p^{-1}].
\]
Let $B$ be any affinoid $K$-algebra and let $M$ be a Banach $B$-module. We recall the following definition, due to Buzzard in \cite{buzz_eigen}.
\begin{defi} \label{defi: proj}
We say that $M$ is \emph{projective}, if there is a Banach $B$-module $N$ such that $M \oplus N$ is potentially orthonormizable.
\end{defi}
\begin{prop} \label{prop: proj dev}
The Banach module $\cusp^{\mc U^\circ}_{\underline v,\underline w}$ is a projective $A^\circ$-module. Moreover, the natural specialization morphism
\[
\cusp^{\mc U^\circ}_{\underline v,\underline w} \to \cusp^{\chi^\circ}_{\underline v,\underline w}
\]
is surjective.
\end{prop}
\begin{proof}
Taking the $v_i$'s all equal, we can assume that $\mathfrak{Y}(\underline v)^{\min,\rig}$ is an affinoid. Since $(\eta(\underline v \times \operatorname{id}))_\ast (\underline{\mathfrak w}^{\dagger \chi_{\mathfrak U^\circ}^{\un}}_{\underline v,\underline w} (-D))$ is a small Banach sheaf by Proposition~\ref{prop: proj small ban}, the proposition is proved exactly as \cite[Corollary~8.2.3.2]{AIP}.
\end{proof}
\begin{coro} \label{coro: proj}
Let $\mc U =\Spm(A) \subset \mc W$ be an affinoid, and let $\chi \in \mc U(K)$. The Banach module $\cusp^{\dagger\mc U}_{\underline v, \underline w}$ is a projective $A$-module. Moreover, the natural specialization morphism
\[
\cusp^{\dagger\mc U}_{\underline v, \underline w} \to \cusp^{\dagger\chi}_{\underline v, \underline w}
\]
is surjective.
\end{coro}
\begin{proof}
Let $\Spf(A^\circ)$ be a formal model of $\mc U^\circ$, the image of $\mc U$ in $\mathfrak W(\underline w)^{\rig}$ and let $\cusp^{\mc U^\circ}_{\underline v,\underline w}$ be as above. Let $\chi^{\un}$ and $\chi^{\un,\circ}$ be the characters associated to $\mc U$ and $\mc U^\circ$. By Remark~\ref{rmk: dev is dev}, we have
\[
\cusp^{\dagger\mc U}_{\underline v, \underline w} = (\cusp^{\mc U^\circ}_{\underline v,\underline w} \otimes_{A^\circ} A(-\chi^{\un,\circ '}))^{\B^{\mc O}(\Z/p^n \Z)},
\]
in particular, by Proposition~\ref{prop: proj dev}, we have that $\cusp^{\dagger\mc U}_{\underline v, \underline w}$ is a direct factor of a projective $A$-module, and hence it is projective. An analogous equality holds for $\cusp^{\dagger\chi}_{\underline v, \underline w}$ and the specialization morphism $\cusp^{\mc U^\circ}_{\underline v,\underline w} \to \cusp^{\chi^\circ}_{\underline v,\underline w}$ is surjective by Proposition~\ref{prop: proj small ban}. Since taking invariants with respect to a finite group is an exact functor on the category of $\Q_p$-vector spaces, the morphism $\cusp^{\dagger\mc U}_{\underline v, \underline w} \to \cusp^{\dagger\chi}_{\underline v, \underline w}$ is surjective as required.
\end{proof}
\subsection{The Eigenvariety} \label{subsec: eigen}
Let $\mc Z(\underline v, \underline w)\subseteq \mc W(\underline w) \times \m A^{1,\rig}$ be the spectral variety associated to the $\U$-operator acting on $\cusp^{\dagger \mc W(\underline w)}_{\underline v, \underline w}$.
\begin{teo} \label{teo: eigen}
There is a rigid space $\mc E_{\underline v, \underline w}$ equipped with a finite morphism $\mc E_{\underline v, \underline w} \to \mc Z(\underline v, \underline w)$ that satisfies the following properties.
\begin{enumerate}
 \item \label{en: basic propr} It is equidimensional of dimension $\rk(\T^{\mc O})$. The fiber of $\mc E_{\underline v, \underline w}$ above a point $\chi \in \mc W(\underline w)$ parametrizes systems of eigenvalues for the Hecke algebra $\m T$ appearing in $\cusp^{\dagger\chi}_{\underline v,\underline w}$ that are of finite slope for the $\U$-operator. If $x \in \mc E_{\underline v, \underline w}$, then the inverse of the $\U$-eigenvalue corresponding to $x$ is $\pi_2(x)$, where $\pi_2$ is the induced map $\pi_2 \colon \mc E_{\underline v, \underline w} \to \m A^{1,\rig}$. For various $\underline v$ and $\underline w$, these construction are compatible. Letting $\underline v \to 0$ and $\underline w \to \infty$ we obtain the global eigenvariety $\mc E$.
 \item \label{en: unr} Let $f \in \cusp^{\dagger\chi}_{\underline v,\underline w}$ be a cuspidal eigenform of finite slope for the $\U$-operator and let $x_f$ be the point of $\mc E_{\underline v, \underline w}$ corresponding to $f$. Let us suppose that $\mc E_{\underline v, \underline w} \to \mc W(\underline w)$ is unramified at $x_f$. Then there exists an affinoid $\mc U \subset \mc W(\underline w)$ that contains $\chi$ and such that $f$ can be deformed to a family of finite slope eigenforms $F \in \cusp^{\dagger\mc U}_{\underline v,\underline w}$.
\end{enumerate}
\end{teo}
\begin{proof}
By Corollary~\ref{coro: proj} all the assumptions of \cite{buzz_eigen} are satisfied, we can then use Buzzard's machinery to construct the eigenvariety, this gives $\mc E_{\underline v, \underline w}$ and \eqref{en: basic propr} follows. Point \eqref{en: unr} is an automatic consequence of the way we used to construct the eigenvariety, see \cite[Proposition~8.1.2.6]{AIP}.
\end{proof}
\begin{coro} \label{coro: fin}
Let $f \in \cusp^{\dagger\chi}_{\underline v,\underline w}$ be a cuspidal eigenform of finite slope for the $\U$-operator. Then there exists an affinoid $\mc U \subset \mc W(\underline w)$ that contains $\chi$ and such that the system of eigenvalues associated to $f$ can be deformed to a family of systems of eigenvalues appearing in $\cusp^{\dagger\mc U}_{\underline v,\underline w}$.
\end{coro}

\section{Classicity results} \label{sec: class}
In this section we prove certain classicity results for our modular forms. Since the results of \cite{BGG} hold in general, there are no conceptual difficulties to generalize \cite{AIP} to our setting, so we will sometimes only sketch the arguments.

Recall that at the end of Section~\ref{sec: PEL data} we have defined the algebraic group $\GL^{\mc O}$ with maximal split torus $\T^{\mc O}$. We also have the Borel subgroup $\B^{\mc O}$ with unipotent radical $\U^{\mc O}$.

Let $\chi \in X^\ast(\T^{\mc O})^+$ be a (classical) dominant weight, that we see as a character of $\B^{\mc O}$ trivial on $\U^{\mc O}$. We set
\begin{gather*}
V_\chi \colonequals \{f \colon \GL^{\mc O} \to \mathds A^1 \mbox{ morphisms of schemes such that, for all } \Z_p-\mbox{algebras } R,\\
\mbox{ we have } f(gb)=\chi(b)f(g) \mbox{ for all } (g,b) \in (\GL^{\mc O} \times \B^{\mc O})(R)\}.
\end{gather*}
There is a left action of $\GL^{\mc O}$ on $V_\chi$ given by $(g \cdot f)(x) = f(g^{-1}x)$.

Recall the sheaf $\underline \omega^\chi$ on $Y_{\Iw}(p)$ of classical modular forms. We have the toroidal compactification $Y_{\Iw}(p)^{\tor}$ of $Y_{\Iw}(p)$ defined in \cite{ben_these}. We will work only with the generic fiber $\mathfrak Y_{\Iw}(p)^{\tor,\rig}$ of its formal completion $\mathfrak Y_{\Iw}(p)^{\tor}$. We have that $\mathfrak Y_{\Iw}(p)(\underline v)^{\tor,\rig}$ is an open subspace of $\mathfrak Y_{\Iw}(p)^{\tor,\rig}$. Clearly $\underline \omega^\chi$ extends to a sheaf on $Y_{\Iw}(p)^{\tor}$.

We have the following
\begin{prop} \label{prop: omega repr class}
Locally for the étale topology on $Y_{\Iw}(p)^{\tor,\rig}$ the sheaf $\underline \omega^\chi$ is isomorphic to $V_{\chi',K}$. This isomorphism respects the action of $\GL^{\mc O}$.
\end{prop}
We write $\B^{\mc O, \op}$ and $\U^{\mc O, \op}$ for the opposite subgroups of $\B^{\mc O, \op}$ and $\U^{\mc O, \op}$. Let $\I^{\mc O}$ be the Iwahori subgroup of $\GL^{\mc O}(\Z_p)$ given by matrices that are `upper triangular' modulo $p$ (recall that we are in the unramified case, so $p$ is a uniformizer of each $\mc O_i$) and let $\N^{\mc O, \op}$ be the subgroup of $\U^{\mc O, \op}(\Z_p)$ given by those matrices that reduce to the identity modulo $p$. We have an isomorphism of groups
\[
\N^{\mc O, \op} \times \B^{\mc O}(\Z_p) \to \I^{\mc O}
\]
given by Iwahori decomposition.

We use the following identification, in case (A) and (C) respectively.
\begin{gather*}
\N^{\mc O, \op} = \prod_{i=1}^k \left( p\mc O_i^{\frac{a_i^+(a_i^+-1)}{2}} \times p\mc O_i^{\frac{a_i^-(a_i^--1)}{2}} \right) \subset \prod_{i=1}^k \left( \mathds A^{\frac{a_i^+(a_i^+-1)}{2}, \rig} \times \mathds A^{\frac{a_i^-(a_i^--1)}{2}, \rig} \right),\\
\N^{\mc O, \op} = \prod_{i=1}^k p\mc O_i^{\frac{a_i(a_i-1)}{2}} \subset \prod_{i=1}^k \mathds A^{\frac{a_i(a_i-1)}{2}, \rig}.
\end{gather*}
Given $\underline w$, a tuple of positive real numbers as in the definition of $\mc W(\underline w)$, we define in case (A) and (C) respectively
\begin{gather*}
\N^{\mc O, \op}_{\underline w} \colonequals \bigcup_{(x_i^\pm) \in \N^{\mc O, \op}} \prod_{i=1}^k \left( B(x_i^+, p^{-w_i^+}) \times B(x_i^-, p^{-w_i^-}) \right),\\
\N^{\mc O, \op}_{\underline w} \colonequals \bigcup_{(x_i) \in \N^{\mc O, \op}} \prod_{i=1}^k B(x_i, p^{-w_i}),
\end{gather*}
where $B(x,p^{-w})$ is the ball of center $x$ and radius $p^{-w}$.

We say that a function $f \colon \N^{\mc O, \op} \to K$ is \emph{$\underline w$}-analytic if it is the restriction of an analytic function $f \colon \N^{\mc O, \op}_{\underline w} \to K$. Note that the extension of $f$, if it exists, is necessarily unique. We write $\mc F^{\underline w-\an}(\N^{\mc O, \op},K)$ for the set of $\underline w$-analytic functions. If $w_i^\pm = 1$ for all $i$ and $f$ is $\underline w$-analytic, we simply say that $f$ is analytic and we write $\mc F^{\an}(\N^{\mc O, \op},K)$ for the set of analytic functions. A function is \emph{locally analytic} if it is $\underline w$-analytic for some $\underline w$ and we write $\mc F^{\locan}(\N^{\mc O, \op},K)$ for the set of locally analytic functions.

Let now $\chi \in \mc W(\underline w)(K)$ be a $\underline w$-analytic character. We set
\begin{gather*}
V_\chi^{\underline w-\an} \colonequals \{f \colon \I^{\mc O} \to K \mbox{ such that } f(ib)=\chi(b)f(i)\\
\mbox{ for all } (i,b) \in \I^{\mc O} \times \B^{\mc O} \mbox{ and } f_{|\N^{\mc O, \op}_{\underline w}} \in \mc F^{\underline w-\an}(\N^{\mc O, \op},K)\}.
\end{gather*}
The definition of the spaces $V_\chi^{\an}$ and $V_\chi^{\locan}$ is similar. They all are representations of $\I^{\mc O}$ via $(i \star f)(x) = f(xi)$.

We have the following
\begin{prop} \label{prop: omega repr gen}
Locally for the étale topology on $\mathfrak Y_{\Iw}(p)^{\tor,\rig}$ the sheaf $\underline \omega^{\dagger\chi}_{\underline v,\underline w}$ is isomorphic to $V_{\chi'}^{\underline w-\an}$. This isomorphism respects the action of $\I^{\mc O}$.
\end{prop}
If $\chi$ is a classical dominant weight, there is an obvious inclusion $V_\chi \hookrightarrow V_\chi^{\underline w-\an}$ and we have the following proposition (see \cite[Proposition~5.3.4]{AIP}).
\begin{prop} \label{prop: incl omega repr}
The open immersion of Proposition~\ref{prop: class mod forms} induces an inclusion of sheaves
\[
\underline \omega^\chi_{|\mathfrak Y_{\Iw}(p)^{\tor,\rig}} \hookrightarrow \underline \omega^{\dagger\chi}_{\underline v,\underline w}.
\]
Locally for the étale topology on $\mathfrak Y_{\Iw}(p)^{\tor,\rig}$ this inclusion is isomorphic to $V_{\chi'} \hookrightarrow V_{\chi'}^{\underline w-\an}$.
\end{prop}
\subsection{The BGG resolution and a classicity result}
The goal of this subsection is to characterize the image of the inclusion $V_\chi \hookrightarrow V_\chi^{\locan}$.

Let $W$ be the Weyl group of $\GL^{\mc O}$ and let $\mathfrak{gl}^{\mc O}$ and $\mathfrak h^{\mc O}$ be the Lie algebras of $\GL^{\mc O}$ and $\T^{\mc O}$ respectively. Let $\Phi \subset X^\ast(\T^{\mc O})$ be the set of roots of $\mathfrak{gl}^{\mc O}$. We have a decomposition
\[
\mathfrak{gl}^{\mc O} = \mathfrak h^{\mc O} \oplus \bigoplus_{\alpha \in \Phi} \mathfrak{gl}^{\mc O}_\alpha.
\]
Let $\Phi^+ \subseteq \Phi$ be the set of positive roots given by the choice of $\B^{\mc O}$ and let $\alpha \in \Phi^+$. We fix an element $0 \neq e_\alpha \in \mathfrak{gl}^{\mc O}_\alpha$. There is an element, that we fix once and for all, $f_\alpha \in \mathfrak{gl}^{\mc O}_{-\alpha}$ such that $\langle e_\alpha, f_\alpha,h_\alpha\rangle$ is isomorphic to $\mathfrak{sl}_2$ via $e_\alpha \mapsto \begin{pmatrix} 0 & 1 \\ 0 & 0 \end{pmatrix}$, $f_\alpha \mapsto \begin{pmatrix} 0 & 0 \\ 1 & 0 \end{pmatrix}$, and $h_\alpha \mapsto \begin{pmatrix} 1 & 0 \\ 0 & -1 \end{pmatrix}$, where $h_\alpha \colonequals [e_\alpha,f_\alpha]$. We denote by $s_\alpha \in W$ the reflection $\chi \mapsto \chi - \langle\chi, \alpha^\vee\rangle \alpha$, where $\chi \in X^\ast(\T^{\mc O})$ and $\alpha^\vee$ is the coroot associated to $\alpha$. The pairing $\langle \cdot, \cdot \rangle \colon X^\ast(\T^{\mc O}) \times X_\ast(\T^{\mc O}) \to \Z$ is the natural one. Given $w \in W$ and $\chi \in X^\ast(\T^{\mc O})$, we set $w \bullet \chi \colonequals w(\chi + \rho) - \rho$, where $\rho$ is half of the sum of the positive roots.

Let now $\chi \in X^\ast(\T^{\mc O})^+$ be a dominant weight. By the results of \cite{BGG} there is an exact sequence
\begin{equation} \label{eq: BGG}
0 \to V_\chi \to V_\chi^{\an} \to \bigoplus_{\alpha \in \Phi^+} V_{s_\alpha \bullet \chi}^{\an}
\end{equation}
where the first map is the natural inclusion and the second one is given by the maps $\Theta_\alpha \colon V_\chi^{\an} \to V_{s_\alpha \bullet \chi}^{\an}$ defined as follows. Recall the action of $\I^{\mc O}$ on $V_\chi^{\an}$ given by $(i \star f)(x) = f(xi)$. It induces, by differentiation, an action of the universal enveloping algebra $U(\mathfrak{gl}^{\mc O})$ on $V_\chi^{\an}$. We set
\[
\Theta_\alpha(f) \colonequals f_\alpha^{\langle \chi, \alpha^\vee \rangle + 1 } \star f.
\]
We now fix an index $i=1,\ldots,k$. Let $1 \leq j < a_i^\pm$. We write $d_{i,j}^\pm \in \GL^{\mc O}(\Q_p)$ for the element whose components are the identity matrix except the $i^\pm$-th one, that is the matrix
\[
\begin{pmatrix}
p^{-1}\mathrm{Id}_{a_i^\pm - j} & 0\\
0 & \mathrm{Id}_j
\end{pmatrix}
\]
If $\chi \in X^\ast(\T^{\mc O})^+$ is a dominant weight, there is an operator $\delta_{i,j}^\pm$ on $V_\chi$ defined by the formula $(\delta_{i,j}^\pm \cdot f)(x)= f(d_{i,j}^\pm x (d_{i,j}^\pm)^{-1})$.

Let now $\chi \in \mc W(\underline w)$ be a $\underline w$-analytic character. Take $f \in V_\chi^{\underline w-\an}$ and $i \in \I^{\mc O}$. By the existence of the Iwahori decomposition there are $b \in \B^{\mc O}(\Z_p)$ and $n \in \N^{\mc O,\op}$ such that $i = nb$. We define an operator $\delta_{i,j}^\pm$ on $V_\chi^{\underline w-\an}$ via the formula
\[
(\delta_{i,j}^\pm \cdot f)(i)= f(d_{i,j}^\pm n (d_{i,j}^\pm)^{-1} b).
\]
Taking restriction to $\N^{\mc O,\op}$ we can identify $V_\chi^{\underline w-\an}$ with $\mc F^{\underline w-\an}(\N^{\mc O,\op},K)$ and under this identification the operator $\prod_j \delta_{i,j}^\pm$ increases the radius of analyticity in the $i^\pm$-th direction.

We use now the notation of Subsection~\ref{subsub: Uj} and we explain the relation between the operators $\U_{i,j}^\pm$ and $\delta_{i,j}^\pm$. We have the following proposition (see \cite[Proposition~6.2.3.1]{AIP}).
\begin{prop} \label{prop: heck op repr}
Let $x,y \in \mathfrak Y_{\Iw}(p)(\underline v)^{\rig}$ such that $y \in p_2(p_1^{-1}(x))$. We fix $\chi \in \mc W(\underline w)$ a $\underline w$-analytic character. Then there is a commutative diagram whose vertical arrows are isomorphisms
\[
\xymatrix{
(\underline \omega^{\dagger\chi}_{\underline v,\underline w})_y \ar[r]^-{(f^\ast)^{-1}} & (\underline \omega^{\dagger\chi}_{\underline v,\underline w})_x \\
V_{\chi'}^{\underline w-\an} \ar[r]^-{\delta_{i,j}^\pm} \ar[u]^-{\wr} & V_{\chi'}^{\underline w-\an} \ar[u]^-{\wr}
}
\]
\end{prop}
Let $\underline \nu = (\nu_{i,j}^\pm)_{i,j}$ be a tuple of positive real numbers such that
\[
\underline \nu \in \prod_{i=1}^k (\mathds R^{a_i^+-1} \times \mathds R^{a_i^--1})
\]
in case (A) and
\[
\underline \nu \in \prod_{i=1}^k \mathds R^{a_i-1}
\]
in case (C). Given $\chi \in \mc W(\underline w)$ we write $V_\chi^{\underline w-\an, < \underline \nu}$ for the intersection of the generalized eigenspace where each $\delta_{i,j}^\pm$ acts with eigenvalues of valuation smaller than $\nu_{i,j}^\pm$. (Here, for a given $i$, we take $j$ in the range $1,\ldots,a_i^\pm - 1$.)
\begin{prop} \label{prop: class sheaf}
Let $\chi \in X^\ast(\T^{\mc O})^+$ be a classical dominant weight. We can write $\chi=(k_{i,s,j}^\pm)_{i,s,j}$ as in Subsection~\ref{subsec: class mod form}. For each $i=1,\ldots,k$ and for each $j=1,\ldots, a_i^\pm - 1$ we set $\nu_{i,a_i^\pm - j}^\pm \colonequals \inf_s \{ k_{i,s,j}^\pm - k_{i,s,j+1} + 1\}$. We then have
\[
V_\chi^{\underline w-\an, < \underline \nu} \subset V_\chi.
\]
\end{prop}
\begin{proof}
We can  fix an index $i$ and work only in the `$i$-th component'. Taking the base change of all objects to $\mc O_i^\pm$ (that makes everything split) and using the exact sequence \eqref{eq: BGG}, the proof is completely analogous to that of \cite[Proposition~2.5.1]{AIP}.
\end{proof}
\subsection{Classicity at the level of sheaves}
Let $\chi = (k_{i,s,j}^\pm)_{i,s,j} \in X^\ast(\T^{\mc O})^+$. We now give a characterization of the image of the inclusion $\underline \omega^\chi_{|\mathfrak Y_{\Iw}(p)^{\tor,\rig}} \hookrightarrow \underline \omega^{\dagger\chi}_{\underline v,\underline w}$ of Proposition~\ref{prop: incl omega repr}.

We can construct a `relative version' over $\mathfrak Y_{\Iw}(p)(\overline v)^{\tor,\rig}$ of the exact sequence \eqref{eq: BGG}. Using then Propositions~\ref{prop: omega repr class}, \ref{prop: omega repr gen}, and \ref{prop: incl omega repr} we obtain the exact sequence of sheaves on $\mathfrak Y_{\Iw}(p)(\overline v)^{\tor,\rig}$
\begin{equation} \label{eq: BGG rel}
0 \to \underline \omega^\chi_{|\mathfrak Y_{\Iw}(p)^{\tor,\rig}} \to \underline \omega^{\dagger\chi}_{\underline v,\underline w} \to \bigoplus_{\alpha \in \Phi^+} \underline \omega^{\dagger s_\alpha \bullet \chi}_{\underline v,\underline w}
\end{equation}
Let $\underline \nu = (\nu_{i,j}^\pm)_{i,j}$ be a tuple of positive real numbers such that $\underline \nu \in \prod_{i=1}^k (\mathds R^{a_i^+} \times \mathds R^{a_i^-})$ in case (A) and $\underline \nu \in \prod_{i=1}^k \mathds R^{a_i}$ in case (C). In case (A) we assume that, for each $i=1,\ldots,k$, we have $\nu_{i,a_i^+}^+ = \nu_{i,a_i^-}^-$. This is a natural condition since $\U_{i,a_i^+}^+ = \U_{i,a_i^-}^-$. We can then write $\nu_{i,a_i^\pm}^\pm \colonequals \nu_{i,a_i^+}^+ = \nu_{i,a_i^-}^-$.

We define $\M^{\dagger \chi, < \underline \nu}_{\underline v,\underline w}$ to be the intersection of the generalized eigenspace where each $\U_{i,j}^\pm$ acts with eigenvalues of valuation smaller than $\nu_{i,j}^\pm$. (The difference with the above situation is that we add the condition the the $\U_{i,a_i^\pm}^\pm$-operator acts with finite slope.)

We now suppose that, for each $i$ and for each $j=1,\ldots, a_i^\pm$ we have $\nu_{a_i^\pm - j}^\pm = \inf_s \{k_{i,s,j}^\pm - k_{i,s,j+1} + 1\}$. As in \cite[Proposition~7.3.1]{AIP} we have the following
\begin{prop} \label{prop: class sheaves}
We have the inclusion
\[
\M^{\dagger \chi, < \underline \nu}_{\underline v,\underline w} \subset \Homol^0(\mathfrak Y_{\Iw}(p)(\underline v)^{\tor,\rig},\underline \omega^\chi).
\]
\end{prop}
\subsection{The classicity theorem} \label{subsec: class thm}
Let $\chi = (k_{i,s,j}^\pm)_{i,s,j} \in X^\ast(\T^{\mc O})^+$ be a classical dominant weight.
\begin{teo} \label{thm: class}
Let $\underline \nu$ be as in the Proposition~\ref{prop: class sheaves}. We moreover assume that, for each $i=1,\ldots,k$, we have
\[
\nu_{i,a_i^\pm}^\pm = \inf_{1 \leq s \leq d_i} (k_{i,s,a_i^+}^+ + k_{i,s,a_i^-}^-) - d_ia_i^+a_i^-
\]
in case (A) and
\[
\nu_{i,a_i} = \inf_{1 \leq s \leq d_i} k_{i,s,a_i} - \frac{d_ia_i(a_i+1)}{2}
\]
in case (C). Then we have an inclusion
\[
\M^{\dagger \chi, < \underline \nu}_{\underline v,\underline w} \subset \M^{\chi},
\]
hence any locally analytic overconvergent modular form in $\M^{\dagger \chi, < \underline \nu}_{\underline v,\underline w}$ is classical.
\end{teo}
\begin{proof}
Using Proposition~\ref{prop: class sheaf}, this follows by the main results of \cite{stefan}.
\end{proof}

\bibliographystyle{amsalpha}
\bibliography{biblio}
\end{document}